\documentclass[12pt,a4paper]{article}%
\usepackage{amsmath}
\usepackage{amssymb}
\usepackage{makeidx}
\usepackage{amsfonts}
\usepackage{graphicx}
\usepackage{pictex}
\usepackage{verbatim}

\providecommand{\U}[1]{\protect\rule{.1in}{.1in}}

\newtheorem{theorem}{Theorem}[section]

\newtheorem{lemma}[theorem]{Lemma}
\newtheorem{proposition}[theorem]{Proposition}

\newtheorem{definition}[theorem]{Definition}
\newtheorem{example}[theorem]{Example}
\newtheorem{remark}[theorem]{Remark}
\newenvironment{proof}[1][Proof]{\noindent\textbf{#1.} }{\hfill \rule{0.5em}{0.5em}}

\usepackage[english,polish]{babel}
\usepackage[colorlinks,citecolor=blue,urlcolor=blue,bookmarks=true]{hyperref}
\hypersetup{
pdfpagemode=UseNone,
pdfstartview=FitH,
pdfdisplaydoctitle=true,
pdfborder={0 0 0}, % No link borders
pdftitle={Oscillating heat kernels},
pdfauthor={Alexander Bendikov, Wojciech Cygan, Wolfgang Woess},
pdflang=en-US
}

\begin{document}

\title{Oscillating heat kernels on ultrametric spaces }
\author{Alexander Bendikov
\and Wojciech Cygan
\and Wolfgang Woess}
\date{}
\maketitle
\begin{center}
\textit{Dedicated to Alexeander A. Grigor'yan\\ on the occasion
of his 60th birthday}
\end{center}
\begin{abstract}
Let $(X,d)$ be a proper ultrametric space. Given a measure
$m$ on $X$ and a function $B \mapsto C(B)$ defined on the collection 
of all non-singleton balls $B$ of $X$, we consider the associated hierarchical Laplacian 
$L=L_{C}\,$. The operator $L$ acts in $\mathcal{L}^{2}(X,m),$ is essentially self-adjoint 
and has a pure point spectrum. 
%Under mild assumptions 
It %s Markov semigroup $(e^{-tL})_{t>0}$ 
admits a continuous heat kernel $\mathfrak{p}(t,x,y)$ with respect to $m$. We 
consider the case when $X$ has a transitive group of isometries under which the operator 
$L$ is invariant and study the 
asymptotic behaviour of the function $t\mapsto \mathfrak{p}(t,x,x)=\mathfrak{p}(t)$.
%which does not depend on $x.$
It is completely monotone, but does not vary regularly. 
When $X=\mathbb{Q}_{p}\,$, the ring of
$p$-adic numbers, and $L=\mathcal{D}^{\alpha} $, the operator of \ fractional
derivative of order $\alpha,$ we show that $\mathfrak{p}(t)=t^{-1/\alpha}\mathcal{A}%
(\log_{p}t)$, where $\mathcal{A}(\tau)$ is a continuous 
non-constant $\alpha$-periodic function. We also study asymptotic behaviour 
of $\min\mathcal{A}$ and $\max\mathcal{A}$ as the space parameter $p$ tends to $\infty$. 
When $X=S_{\infty}\,$, the infinite symmetric group, and $L$ is a hierarchical 
Laplacian with metric structure analogous to $\mathcal{D}^{\alpha},$  we show that, 
contrary to the previous case, the completely monotone function $\mathfrak{p}(t)$
oscillates between two functions $\psi(t)$ and $\Psi(t)$ such that
$\psi(t)/\Psi(t)\to 0$ as $t \to \infty\,$.{\let\thefootnote\relax\footnote
{2010 \emph{Mathematics Subject Classification.} 
60J35;  %Transition functions, generators and resolvents
12H25, %p-adic differential equations 
20K25, %Direct sums, direct products, etc.
%35S05, %Pseudodifferential operators
%47D07, %Markov semigroups and applications to diffusion processes
47S10, %Operator theory over fields other than R, C or the quaternions; non-Archimedean operator theory 
60J25. %Continuous-time Markov processes on general state spaces 

\emph{Key words and phrases.} Ultrametric space, hierarchical Laplacian, 
isotropic Markov semigroup, heat kernel, return probabilities, oscillations.

Supported by Austrian Science Fund project FWF P24028, NAWI Graz, and National 
Science Centre, Poland, grants 2013/11/N/ST1/03605 and 2015/17/B/ST1/00062.
}}

\end{abstract}

%\keywords{}
%\tableofcontents

\section{Introduction}\label{sec:intro}

\setcounter{equation}{0} At the centre of this paper stands the
study of the on-diagonal asymptotics of the heat kernels for
isotropic Markov processes on homogeneous ultrametric spaces.
We focus on the precise computation of its periodic oscillations
and the corresponding background.

Motivation for this paper comes from various sides. One is the interest
in properties of random walks  on infinite
groups, in particular asymptotics of transition probabilities. 
While there is a great body of rigorous mathematical work
on this subject regarding finitely generated groups, here we have drawn 
inspiration from the work on \emph{infinitely generated} ones. This goes back to
Darling and Erd\"os \cite{DarErd} and was further pursued by various authors,
among which Flatto and Pitt \cite{FlatPitt}, Cartwright \cite{Cartwright},
Lawler \cite{Lawler} and Brofferio and Woess \cite{BrofWoess2001}.
Most notably, \cite{Cartwright} is to our knowledge the first to have
exhibited evidence of oscillatory behaviour for return probabilities in
such a case. 
A similar phenomenon was found for random walks on the infinite graph
versions of certain fractals: on the Sierpinski-graphs 
in arbitrary dimension, precise asymptotics and periodic oscillations 
were exhibited by Grabner and Woess \cite{GraWoess}, \cite[\S 16]{Wbook},   
later generalised to a larger class of self-similar graphs in \cite{KroenTeufl}.

The above-mentioned random walks on infinitely generated groups 
can be described as isotropic Markov processes on ultrametric spaces in discrete as
well as continuous time. In Theoretical Physics, such spaces are also known
as \emph{hierarchical lattices,} proposed by F. J. Dyson in his famous paper 
on the phase transition for the $\mathbf{1D}$ ferromagnetic model with long range
interaction~\cite{Dyson2}.
The corresponding \emph{hierarchical Laplacian} $L$ was studied in a variety  
of papers, see e.g.~\cite{Bovier},
\cite{AlbeverioKarwowski}, \cite{Kritchevski1}, \cite{Kritchevski2},
%\cite{Kritchevski3}, 
\cite{Kozyrev}, \cite{Krutikov1}, %\cite{Krutikov2},
%\cite{AisenmanMolchanov}, 
\cite{Molchanov}, \cite{Rodriges-Zuniga}, \cite{Dellacherie2009}. It is the generator
of a Markov process, which in the situation of discrete time and space is
a random walk on an inifinitely generated group as mentioned above.

A systematic study of isotropic Markov semigroups defined on an ultrametric
measure space $(X,d,m)$ has been carried out by Bendikov, Grigor'yan and Pittet 
 in~\cite{BGP} (where $X$ is discrete) and by the same authors plus Woess in~\cite{BGPW} 
(where $X$ may contain both isolated and non-isolated points). The approach
of those papers puts the previous work into a well accessible framework which 
has lead to a good unterstanding and various new results. One of the key features
is that for the Markov generator (that is, the hierarchical Laplacian), 
one obtains a full description
of the spectrum, which is pure point, along with a complete system of compactly
supported eigenfunctions.

On the basis of this spectral decomposition, we study here the asymptotics
of the diagonal elements of the heat kernel (transition kernel).
We now display the setup and basic features in more detail. 

\bigskip

\noindent
\textbf{General setup and basic facts}

\medskip

The state space of our processes is assumed to be a proper 
%and separable --> this follows from properness, as well as local compactness
ultrametric space $(X,d)$. 
Recall that proper means that closed balls are compact.
Our metric $d$ must satisfy the ultrametric inequality
\begin{equation}
d(x,y)\leq\max\{d(x,z),d(z,y)\}.
\end{equation}
A basic consequence %of the ultrametric property 
is that each open ball is compact. Furthermore, for each $x \in X$, the set of 
distances $\{d(x,y) : y \in X\}$ is countable and does not accumulate in
$(0\,,\,\infty)$. Two balls are either disjoint or one is contained in the
other. The collection of all balls with a fixed positive radius forms
a countable partition of $X$, and decreasing the radius leads to a refined
partition: this is precisely the structure of a ``hierarchical lattice''
as in the old papers, going back to \cite{Dyson2}.

\emph{The analysis undertaken here is of interest in the case when
$X$ is non-compact, which will be assumed throughout this paper.}

Our setup is based on the following two ingredients.
The first is a Radon measure $m$ on $X$ such that $m(X)= \infty$ and $m(B)>0$ for each closed
ball which is not a singleton, and $m(\{a\})>0$ if and only if $a$ is an isolated
point of $X$.
Now let $\mathcal{B}$ be the collection of all balls with $m(B) > 0$. Then each $B \in \mathcal{B}$
has unique \emph{predecessor} or \emph{parent} $B' \in \mathcal{B} \setminus \{B\}$
which contains $B$ and is such that $B \subseteq D \subseteq B'$ for $D \in \mathcal{B}$
implies $D \in \{ B, B'\}$. In this case, $B$ is called a \emph{successor} of $B'$. 
By properness of $X$,
each non-singleton ball has only finitely many (and at least 2) successors. Their number is
the \emph{degree} of the ball.

The second ingredient is a \emph{choice function}  $C:\mathcal{B} \rightarrow(0\,,\,\infty)$ 
which must satisfy, for all $B\in\mathcal{B}$ and
all non-isolated $a\in X$,
\begin{equation}\label{eq:C1condition}%
\begin{aligned}
\lambda(B)&=\sum\limits_{D\in\mathcal{B}\,:\,D\supseteq B}C(D)<\infty\,, \quad\text{and}
\\
\lambda(\{a\})&=\sum_{B\in\mathcal{B}\,:\,B \ni a}C(B)=\infty\,.
\end{aligned}
\end{equation}
Let $\mathcal{F}$ be the set of all locally constant functions having compact
support. Given the space $X$, the measure $m$ and the choice function $C$, 
we define (pointwise) the hierarchical Laplacian $L_{C}\,$: 
for each $f$ in $\mathcal{F}$ and $x\in X$,
\begin{equation}\label{eq:hlaplacian}
L_{C}f(x):=\sum\limits_{B\in\mathcal{B}\,:\, B \ni x} C(B)
\left(f(x)-\frac{1}{m(B)}\int_{B}f\,dm\right). 
\end{equation}
The present work builds upon the following facts concerning $L_C$ and its
spectrum, which were proved in \cite{BGP} (discrete space) and in full generality
in \cite{BGPW}. The operator $(L_{C},\mathcal{F})$ acts in 
$\mathcal{L}^{2}=\mathcal{L}^{2}(X,m)$. It
is symmetric and admits an $\mathcal{L}^{2}$-complete system of eigenfunctions 
$\{ f_B : B \in \mathcal{B}\}$, % in $\mathcal{F}$, 
given by
\begin{equation}
f_{B}=\frac{\mathbf{1}_{B}}{m(B)}-\frac{\mathbf{1}_{B^{\prime}}%
}{m(B^{\prime})}\,. \label{eigenfunction}%
\end{equation}
The eigenvalue %$\lambda(B^{\prime})$ 
corresponding to $f_{B}$ depends only on $B^{\prime}$ and is $\lambda(B^{\prime})$,
as given in \eqref{eq:C1condition}.
Since all $f_{B}$ belong to $\mathcal{F}$ and the system $\{f_{B} : B \in \mathcal{B}\}$ is
complete in $\mathcal{L}^{2}$ we conclude that $(L_{C},\mathcal{F})$ is an essentially
self-adjoint operator in $\mathcal{L}^{2}.$ By a slight abuse of notation, we shall write
$(L_{C},\mathsf{Dom}_{L_{C}})$ for its unique self-adjoint extension. For all
of this we refer to \cite{BendikovKrupski}, \cite{BGPW}. See also the related
papers %\cite{Kigami}, 
\cite{Kochubey2004}, \cite{Kozyrev},
\cite{PearsonBellisard}.

Observe that to define the functions $C(B),$ $\lambda(B)$ and in particular
the operator $(L_{C},\mathsf{Dom}_{L_{C}})$ we do not need to specify the
ultrametric $d.$ What is needed is the family of all balls which evidently can
be the same for two different ultrametrics $d$ and $d^{\prime}$, via a feasible change of
the diameter function. On the other
hand, given the data $(X,d,m)$ and choosing the function
\begin{equation}\label{standard1}
C(B)=\frac{1}{\mathsf{diam}(B)}-\frac{1}{\mathsf{diam}(B^{\prime})},
\end{equation}
where $B\in \mathcal{B}$ and $B^{\prime}$ is the predecessor of $B$,
we obtain the hierarchical Laplacian $(L_{C},\mathsf{Dom}_{L_{C}})$ having 
eigenvalues of the form
\begin{equation}\label{standard2}
\lambda(B)=\frac{1}{\mathsf{diam}(B)}\,,\quad B \in \mathcal{B}.
\end{equation}
We will refer to the resulting operator $(L_{C},\mathsf{Dom}_{L_{C}})$ as the \emph{standard}
hierarchical Laplacian associated with $(X,d,m).$
Vice versa, even if we start with an \emph{a priori} ultrametric $d$ on $X$, any choice function
satisfying \eqref{eq:C1condition} induces the possibly different \emph{intrinsic} metric $d_*$ via 
\eqref{standard1} and \eqref{standard2}: one sets 
\begin{equation}\label{eq:intrinsic}
\mathsf{diam}_*(B)=\frac{1}{\lambda(B)}\,,\quad B \in \mathcal{B},
\end{equation}
so that for distinct $x,y \in X$,  one has $d_*(x,y) = \mathsf{diam}_*(B)$, where $B$ is the
smallest ball containing both elements. By construction, the collection of $d_*$-balls coincides
with the collection of $d$-balls, so that both metrics induce the same family of hierarchical
Laplacians when varying the choice function~$C$.

The general theory developed in \cite{BGP} and \cite{BGPW}
applies here. In particular, under mild assumptions -- which hold in the 
situations that we consider
here -- the semigroup $(P^{t})$ induced by $L_C$
admits a continuous heat kernel $\mathfrak{p}(t,x,y)$ with respect to $m$. 

It has an integral representation in terms of the \emph{spectral function}
\begin{equation}\label{specfun}
N(x,\tau) = 1 \big/ m\bigl(B_*(x,1/\tau)\bigr)\,,
\end{equation}
where $B_*(x,r) = \{ y\in X : d_*(x,y) \le r \}$, 
see  \cite[Def. 2.8.]{BGPW}. In particular, 
\begin{equation}\label{on-diagonal p}%
\mathfrak{p}(t,x,x)=t\int\limits_{0}^{\infty}N(x,\tau)\exp(-t\tau)\,d\tau\,.
\end{equation}
Since $\tau \mapsto N(x,\tau)$ is a step function, the integral reduces to
an infinite sum.
For the analysis undertaken in this paper, we require the following to hold.

\begin{definition}\label{def:homogeneous}
\rm The ultrametric measure space $(X,d,m)$ and the hierarchical Laplacian $L_{C}$ 
are called \emph{homogeneous,} if there is a group %$\mathfrak{G}$ 
of isometries of $(X,d)$ which
\begin{itemize}
\item acts transitively on $X$, and

\item leaves both the reference measure $m$ and the function $C(B)$ invariant.
\end{itemize}
\end{definition}

The first assumption implies that $(X,d)$ is either discrete or perfect. Basic
examples which we have in mind are

\begin{enumerate}
\item $X=\mathbb{Q}_{p}$ -- the ring of $p$-adic numbers, where $p \ge 2$ (integer).

\item $X = %\mathfrak{G} = 
\bigoplus_{j=1}^{\infty} \mathbb{Z}(p)_j\,
$ --
the direct sum of countably many copies $\mathbb{Z}(p)_j$ of the additive group 
$\mathbb{Z}(p) = \mathbb{Z}/(p\,\mathbb{Z})$. 
%or analogously, $X = \mathbb{Q}_{p}/\mathbb{Z}_{p}\,$.

\item $X=S_{\infty}\,$ -- the infinite symmetric group, that is, the group
of all permutations of the positive integers that fix all but finitely many elements.
\end{enumerate}

In this setting, our main goal is to study the asymptotic behaviour of the function
$t\mapsto \mathfrak{p}(t) = \mathfrak{p}(t,x,x)$ as $t$ tends to $0$ or to $\infty\,$; it does not
depend on $x$ by homogeneity. Our study was
inspired by the results of \cite{Cartwright}, \cite{GraWoess},
\cite{Lawler} and \cite{BrofWoess2001}. The papers \cite{BenBobPit2013} and
\cite{BenLSC2016} are direct forerunners of the present work. See also the related papers \cite{Dawson}, \cite{Dawson2} and \cite{Bojdecki}.

\medskip

Let us describe the main body of the paper. In Section \ref{sec:homogeneous} 
we discuss in more detail
the consequences of the homogeneity assumptions of Definition \ref{def:homogeneous}.
In particular, we show that the on-diagonal heat kernel $\mathfrak{p}(t)$ cannot vary regularly.
In Section \ref{sec:derivative} we consider $X=\mathbb{Q}_{p}\,$, the ring of $p$-adic numbers. 
%endowed with its standard $p$-adic ultrametric $d(x,y)=\left\vert x-y\right\vert_{p}$ 
%and the normed Haar measure $m$. 
We study the operator $\mathfrak{D}^{\alpha}$, the
$p$-adic fractional derivative of order $\alpha>0$. This operator was first 
considered by Taibleson \cite{Taibleson75} as a spectral multiplier on $\mathbb{Q}_{p}$ as
well as on $\mathbb{Q}_{p}^d\,$.
%Dear Sasha, I do not like the fact that you somehow hide Taibleson
%in favour of Vladimirov !
In relation to
the concept of $p$-adic Quantum Mechanics, it was introduced under the above name
by V.S. Vladimirov \cite{Vladimirov}, \cite{VladimirovVolovich}, \cite{Vladimirov94}. The
operator $\mathfrak{D}^{\alpha}$ is the most typical example of a 
homogeneous hierarchical Laplacian. 
We show that the associated on-diagonal heat kernel  on $\mathbb{Q}_{p}$ has the form 
$\mathfrak{p}_{\alpha}(t)=t^{-1/\alpha}\mathcal{A}(\log_{p}t)$, where 
$\mathcal{A}(\tau)$ is a non-constant strictly positive continuous $\alpha$-periodic
function, and that, as $p$ tends to infinity,
$$
\max\mathcal{A}(\tau)\rightarrow(e\alpha)^{-1/\alpha}\quad\text{and}\quad
\min\mathcal{A}(\tau)\sim\frac{(\ln p)^{1/\alpha}}{p}.
$$
In Section \ref{sec:rw} we briefly explain the equivalent approach in terms of isotropic Markov 
processes. In particular, we focus again on the homogeneous situation, where we get
isotropic random walks with discrete as well as continuous time on ultrametric groups. 
We consider there case where the group is the direct sum of a countable familiy of
copies of one finite group.

Finally, in Section \ref{sec:symmetric} we consider a class of homogeneous hierarchical Laplacians 
acting on $\mathcal{L}^{2}(X,m)$, where $X=S_{\infty}$ is the infinite symmetric group
and $m$ is the counting measure.  
%The operator $(L,\mathcal{D})$ is defined as
%the (unique) hierarchical Laplacian whose eigenvalues $\lambda(B)$,
%$B\in\mathcal{B}$, satisfy the equation
%\[
%\lambda(B)=V^{-\delta}\phi(\ln V),
%\]
%where $V=m(B)$, $\phi(\tau)$ varies regularly at $\infty$ (see
%\cite{BinGolTeu}), and $\delta>0$ \ (compare with $\mathfrak{D}^{\alpha}$).
%Let $\mathfrak{p}(t,x,y)$ be the heat kernel associated to $L$. We show that the function
%$\mathfrak{p}(t)=\mathfrak{p}(t,x,x)$ oscillates as $t$ tends to infinity between two functions
%$\psi(t)$ and $\Psi(t)$, where%
%\[
%\psi(t)=t^{-1/\delta}\left(  \frac{\ln x_{\ast}}{\phi(x_{\ast}\ln x_{\ast}%
%)}\right)  ^{1/\delta}\text{ and \ }\Psi(t)=\left(  \frac{1}{e\delta}\right)
%^{1/\delta}t^{-1/\delta}\left(  \frac{x_{\ast}^{\delta}}{\phi(x_{\ast}\ln
%x_{\ast})}\right)  ^{1/\delta}%
%\]
%and
%\[
%x_{\ast}\sim\frac{1}{\delta}\frac{\ln t}{\ln\ln t}\text{ \ at \ }\infty.
%\]
%Notice that in the examples above the reciprocal of the heat kernel on the
%diagonal, i.e. the function $t\rightarrow1/\mathfrak{p}(t)$, is of finite order. It is
We show that the function
$\mathfrak{p}(t)$ oscillates  between two functions
$\psi(t)$ and $\Psi(t)$, where $\Psi(t)/\psi(t) \to \infty$
as $t$ tends to infinity. This case is related to the card shuffling models of
\cite{Lawler} and \cite{BrofWoess2001}; compare with \cite{BGP}.

\section{The homogeneous Laplacian}\label{sec:homogeneous}
\setcounter{equation}{0}

In this section, we discuss general consequences of the homogeneity assumptions.
First of all, the set of distances $\{d(x,y) : y \in X \}$ is the same
for each $x \in X$. Since $X$ is assumed to be non-compact, we have the following
two cases.
\begin{itemize}
 \item[]\emph{Case 1.} $\;X$ is perfect, and
$\{d(x,y) : y \in X \} = \{0\} \cup \{ r_k : k \in \mathbb{Z} \}$,\\
where $r_k < r_{k+1}$ with $\lim\limits_{k\to \infty} r_k = \infty$
and  $\lim\limits_{k\to -\infty} r_k = 0\,$;\\[-14pt]
\begin{equation}\label{eq:cases}
 \quad
\end{equation}
\\[-43pt]
 \item[]\emph{Case 2.} $\;X$ is countable, and 
$\{d(x,y) : y \in X \} = \{ r_k : k \in \mathbb{N}_0 \}$, \\
where $r_0=0$, $r_k < r_{k+1}$ with $\lim\limits_{k\to \infty} r_k = \infty\,$.
\end{itemize}
In both cases, we let 
$$
\mathcal{B}_k = \{ B(x,r_k) : x \in X \}
$$ 
be the collection of all closed balls of diameter $r_k$ in $X$.
This is a partition of $X$, and it is finer than $\mathcal{B}_{k+1}$.
By homogeneity, all balls in $\mathcal{B}_k$ are isometric. In particular,
the number $n_k$ of successor balls is the same for each ball in  $\mathcal{B}_k\,$,
where $k \in \mathbb{Z}$ in Case 1, and $k \in \mathbb{N}_0$ in Case 2.
We call the two- or one-sided infinite sequence $(n_k)$ the
\emph{degree sequence} of $(X,d)$. Note that $2 \le n_k < \infty$.

When we pass from the \emph{a priori} metric $d$ to the intrinsic metric $d_*$
given by \eqref{eq:intrinsic},
the distances $r_k$ are transformed into new ones -- which we shall denote by $s_k$ --
but the collections $\mathcal{B}_k$ and the degree sequence remain the same.

We remark that one can associate an infinite \emph{tree} with $X$. Its vertex set is
$\mathcal{B}$, and there is an edge between any $B \in \mathcal{B}$ and its predecessor
$B'$. In this situation, $\mathcal{B}_k$ is the \emph{horocyle} of the tree with index $-k$,
and $X$ is the (lower) \emph{boundary} of that tree. For more details and figures
in a context close to the one discussed here, see \cite{BGPW}, as well as \cite{Figa-Tal1},
\cite{Figa-Tal2},  \cite{Dellacherie2009} and \cite{BendikovKrupski}. 

Independently of the initial algebraic or geometric model, our homogeneous ultrametric
space $(X,d)$ is uniquely determined by the degree sequence. 
For having homogeneity, the reference measure $m$ is also uniquely defined up to a constant
factor. If we normalise by setting $m(B) = 1$ for each $B \in \mathcal{B}_0$, then
\emph{a fortiori}, for any $k \in \mathbb{Z}$ (Case 1), resp. $k \in \mathbb{N}_0$ (Case 2),  
$$
\begin{aligned}
m(B) &= V(k) \; \text{ for all }\; B \in \mathcal{B}_k\,, \quad\text{where}\quad V(0)=1,\\
%\quad\text{and}\\
V(k) &= \begin{cases} n_1 n_2 \cdots n_k &\text{for $k > 0$, and}\\
                      1/(n_{k+1} n_{k+2} \cdots n_0)&\text{for $k < 0$ in Case 1.}
        \end{cases}
\end{aligned}
$$
This determines $m$ uniquely as a measure on the Borel $\sigma$-algebra of $X$.
%, since $\mathcal{B}$ is a ``good'' generator of that $\sigma$-algebra.
Regarding the hierarchical Laplacian, homogeneity means that $C(B) = C_k$ is 
the same for each $B \in \mathcal{B}_k\,$.

Summarising, we see that any homogeneous ultrametric space plus Laplacian are
completely determined by the degree sequence $(n_k)$ and the sequence $(C_k)$
which defines the homogeneous choice function.

Along with the choice function, also the eigenvalues of 
\eqref{eq:C1condition} depend only on $k$:
$$
\lambda(B) = \lambda_k \; \text{ for all }\; B \in \mathcal{B}_k\,, \quad\text{where}
\quad\lambda_k = \sum_{\ell \ge k} C_{\ell}\,.
$$
Each of them has infinite multiplicity in the pure point spectrum of $L_C\,$, 
see e.g \cite[Section 3.2]{BGPW}.
Regarding the intrinsic metric, recall that
$$
d(x,y) = r_k \iff d_*(x,y) = s_k\,,\quad \text{where}\quad s_k = 1/\lambda_k\,. 
$$
The volume function associated with the intrinsic metric is of course independent of 
$x \in X$, and given by
\begin{equation}\label{eq:V*}
V_*(s) = m\bigl(B_*(x,s)\bigr) = V(k) \quad \text{for}\quad s \in [s_k\,,\,s_{k+1})\,.
\end{equation}
%We note that $V(k+1) = n_{k+1}\,V(k)$.
The spectral function $N(x,\tau) = N(\tau)$ of \eqref{specfun} is also independent of
$x$, with $N(\tau) = 1/V_*(1/s)$. Thus, the formula \eqref{on-diagonal p} for the on-diagonal
heat kernel  becomes
\begin{equation}\label{eq:ondiag-HK}
\mathfrak{p}(t)=\int\limits_{0}^{\infty}\exp(-t\tau)\,dN(\tau)=
{\displaystyle\sum\limits_{k}}
e^{-t\lambda_k}\left(  \frac{1}{V(k-1)}-\frac{1}{V(k)}\right).
\end{equation}
Here and in the sequel, the summation ranges over $k \in \mathbb{Z}$ in Case 1, 
and over $k \in \mathbb{N}$ in Case 2.

\smallskip

The following plays an important role in the context of heat kernel estimates
of many types of Laplacians, not just on ultrametric spaces.
A monotone increasing function $F:\mathbb{R}_{+}\to \mathbb{R}_{+}$ is said 
to satisfy the \emph{doubling property} if there exists a constant $D>0$ such that
$F(2s)\le D\,F(s)$ for all $s>0$.

\begin{proposition}\label{N-doubling} The following properties are equivalent.
\begin{enumerate}
\item[(i)] The function $N(\tau)$ is doubling.

\item[(ii)] There are finite bounds $D$ and $K$ such that for all $k \in \mathbb{Z}$, 
resp $\in \mathbb{N}$, and all $\tau > 0$,
\[
n_{k}\leq D \quad \text{and}\quad \#\{l:\tau\leq\lambda_l\leq2\tau\}\leq K\,.
\]
\end{enumerate}
\end{proposition}

\begin{proof}
(i)$\implies$(ii). Assume that $N(\tau)$ is doubling. Since by definition
$N(\tau)=1/V_{\ast}(1/\tau)$, the function $V_{\ast}(s)$ is doubling as well.
We use \eqref{eq:V*}. Choose
$s_{k}<s<s_{k+1}$ such that $2s>s_{k+1}$, then by the doubling property,%
\[
V_{\ast}(s_{k+1})\leq V_{\ast}(2s)\leq D\,V_{\ast}(s)=D\, V_{\ast}(s_{k}).
\]
It follows that%
$$ %\begin{equation}
n_{k+1}=V_{\ast}(s_{k+1})/V_{\ast}(s_{k})\leq D.
%\label{n-k-upper bound}%
$$%\end{equation}
Let $\tau >0$ and set $s = 1/(2\tau)$. Then there are $k$ and $r$ such that 
$s_{k}\leq s<s_{k+1}$ and $s_{k+r}\leq2s<s_{k+r+1}\,$. We claim that
\[
2^{r}\leq n_{k+1}...n_{k+r}\leq D^{2}.
\]
Indeed, this  is true if $r=1.$ Assuming that $r\geq2$ we obtain%
%\begin{align*}
$$
2^{r}  \leq n_{k+1}...n_{k+r}\leq %Dn_{k+2}...n_{k+r}\\
D\, \frac{V_{\ast}(s_{k+2})}{V_{\ast}(s_{k+1})}\cdots
\frac{V_{\ast}(s_{k+r})}{V_{\ast}(s_{k+r-1})} %=D \, \frac{V_{\ast}(s_{k+r})}{V_{\ast}(s_{k+1})}
= D\,\frac{V_{\ast}(2s)}{V_{\ast}(s)} \leq D^{2}\,.
$$%\end{align*}
It follows that%
$$%\begin{equation}
\#\{l:\tau\leq\lambda_l\leq2\tau\}=\#\{l:s\leq s_{l}\leq2s\}
%\label{tau-distribution}%
%\end{equation}%
%\[
\leq r+1\leq\frac{2\ln D}{\ln2}+1.
$$

\noindent
(ii)$\implies$(i). Assume that $n_{k}\leq D$ for  all
$k\in \mathbb{Z}$, resp. $k\in \mathbb{N}$. For any $s>0,$ let $m(s)= \#\{l:s\leq s_{l}\leq2s\}$. 
By assumption $m(s)\leq K$ for all $s>0$. We have%
\[
V_{\ast}(2s)\leq D^{m(s)}\,V_{\ast}(s)\leq D^{K}\,V_{\ast}(s),
\]
whence $V_{\ast}(\tau)$ is doubling. Since $N(\tau)=1/V_{\ast}(1/\tau)$, the
function $N(\tau)$ is doubling as well.
\end{proof}

\begin{remark}\rm
Recall that  $\mathfrak{p}(t)$ is the Laplace transform of the function
$N(\tau)$. It follows that for all $t > 0$
\begin{equation}\label{eq:DP}
c\, N(1/t) \le \mathfrak{p}(t) \le c'\, N(1/t) \quad \text{for some}\; c, c' > 0
\end{equation}
if and only if the function $N(\tau)$ is doubling; see \cite[Theorem 2.14 \& Lemma 2.21]{BGPW}. 
(This holds also in the non-homogeneopus case.) 

When the function $N(\tau)$ is not doubling, one can state only that setting
$M(\tau) = -\log N(\tau)$, we have for $t \to \infty$
\begin{equation}\label{Laplace asymptotics}%
\log\frac{1}{\mathfrak{p}(t)} \sim M^{\bigstar}(t)\,,% \;\text{ as } t \to \infty\,, 
%\end{equation}
\quad \text{where}\quad
M^{\bigstar}(t) = \inf\{t\tau +M(\tau) : \tau > 0\}
\end{equation}
is the Legendre transform of the function $M(\tau)$. The papers \cite{BenBob} and
\cite{BenBobPit2013} contain many computations based on 
\eqref{Laplace asymptotics}. \hfill\rule{0.5em}{0.5em}
\end{remark}

Following \cite[Section 2.2.2]{BinGolTeu}, a function $F(t)>0$ is of \emph{finite order} 
$\rho>0$~if
\begin{equation}\label{def-finite order}%
\lim_{t\to\infty}\frac{\log F(t)}{\log t}=\rho. 
\end{equation}
Any function of the form $F(t) = a(t)\,t^{\rho}\, \exp\bigl(\log(1+t)\bigr)^{\varepsilon}$,
where $0<\varepsilon<1$ and $a(t)$ is strictly positive and bounded, is of
finite order $\rho$. Two examples related to the heat kernel $\mathfrak{p}(t)$ will be
presented in Sections \ref{sec:derivative} and \ref{sec:rw} of this paper.

\begin{proposition}\label{pro:order} 
The following statements are equivalent.
\begin{enumerate}
\item[(a)] One of the functions $1/\mathfrak{p}(t)\,,\;1/N(1/t)¸\,,\;V_{\ast}(\tau)$ 
is of finite order $\rho.$
\item[(b)] Each of the functions $1/\mathfrak{p}(t)\,,\;1/N(1/t)¸\,,\;V_{\ast}(\tau)$ is of
finite order $\rho.$
\item[(c)] $\;\log V(k)\sim \rho\, \log(1/\lambda_k)\,\;$ as $k \to \infty\,$.\footnote[1]{
Throughout
this paper, $\sim$ denotes asymptotic equivalence, i.e., quotients tend to~$1$.}
\end{enumerate}
\end{proposition}

\begin{proof}
$(a)\Rightarrow(b).\;$  The Abelian part of the statement%
\[
\mathtt{order}\bigl(1/N(1/t)\bigr)) \; =\; \rho
\Rightarrow\mathtt{order}\bigl(1/\mathfrak{p}(t)\bigr)=\rho
\]
follows by a standard argument from the Laplace transform analysis. Thus, what is
left is the the Tauberian part of the statement, that is, the converse implication. 
%\[
%\mathtt{order}(1/\mathfrak{p}(t))=\rho\Rightarrow\mathtt{order}(1/N(1/t))=\rho.
%\]
Set again $M(\tau)=\log(1/N(\tau)).$ By \eqref{def-finite order} and
\eqref{Laplace asymptotics}, for any given $\varepsilon>0$ there exists $T>0$
such that for all $t\geq T,$%
$$
t\tau+M(\tau)  \geq \inf \{t\tau +M(\tau) : \tau > 0\} = M^{\bigstar}(t)
\geq(\rho-\varepsilon)\log t.
$$
In particular, choosing $\tau=1/t$ we obtain%
\[
M(1/t)\geq(\rho-\varepsilon)\, \log t - 1\sim(\rho-\varepsilon)\, \log t.
\]
It follows that
\begin{equation}
\liminf_{t\to\infty}\frac{\log\bigl(1/N(1/t)\bigr)}{\log t}\geq\rho-\varepsilon.
\label{liminfN}%
\end{equation}

Let $\widetilde{M}\leq M$ be any continuous strictly decreasing function 
$\mathbb{R}_+ \to \mathbb{R}_+$ with $\widetilde{M}(0+)=\infty.$ There exists 
$\widetilde T>0$ such that for all
$t\geq \widetilde T$,%
\begin{equation}\label{Legendre M-tilde}%
\widetilde{M}^{\bigstar}(t) \le M^{\bigstar}(t) \leq(\rho+\varepsilon)\log t. 
%\inf_{s}\{st+\widetilde{M}(s)\}\leq\inf_{s}\{st+M(s)\}\leq(\rho+\varepsilon
%)\ln t. \label{Legandre M-tilde}%
\end{equation}
%Since $\widetilde{M}(s)$ is continuous and decreasing, 
For any given $t>0$
there exists a unique $\tau_{t}$ such that $t\, \tau_{t}=\widetilde{M}(\tau_{t})$,
whence
$$
\widetilde{M}^{\bigstar}(t) \ge \min \Bigl\{ \max \bigl\{ t\, \tau\,,\,\widetilde{M}(\tau)\bigr\} 
: \tau > 0 \Bigr\}
= t\,\tau_t\,. %=\widetilde{M}(\tau_t).
$$
This together with \eqref{Legendre M-tilde} implies that%
$$
t \, \tau_t = \widetilde{M}(\tau_t) \leq(\rho+\varepsilon)\,\log t\,. %\label{inequality}%
$$
In turn, this implies that $\tau_{t}\to 0$ as $t \to \infty$, and that for sufficiently
large $t$,
$$
(\rho+\varepsilon)\,\log t  \geq\widetilde{M}\left(  \frac{\rho+\varepsilon}{t}\,\log t\right)
\geq\widetilde{M}\left(  \frac{1}{t^{1-\varepsilon}}\right)\,,
$$
whence, setting $\tau:=t^{1-\varepsilon},$ we obtain%
\[
\widetilde{M}\left(  \frac{1}{\tau}\right)  
\leq\frac{\rho+\varepsilon}{1-\varepsilon}\,\log\tau\,.
\]
As $\widetilde{M}$ was chosen to be any continuous strictly decreasing
function satisfying $\widetilde{M}\leq M,$ it follows that%
$$
\limsup_{\tau\to\infty}\frac{\log\bigl(1/N(1/\tau)\bigr)}{\log\tau}
\leq\frac{\rho+\varepsilon}{1-\varepsilon}. %\label{limsupN}%
$$
This holds for arbitrarily small $\varepsilon>0$, and together with 
\eqref{liminfN}, it leads to the desired result.

\smallskip

\noindent
$(b)\Leftrightarrow(c).\;$ By definition, $V_{\ast}(1/\lambda_k)=V(k),$ whence%
\[
\frac{V_{\ast}(1/\lambda_k)}{\log( 1/\lambda_k)}
=\frac{\log V(k)}{\log( 1/\lambda_k)  }.
\]
Since $\lambda_k\to 0$ as $k \to \infty\,$, the equivalence of (b) and (c) follows.
\end{proof}

\smallskip

At last, recall that a positive function $F(t)$ varies regularly of index $\alpha$ if
\[
\lim_{t\to\infty}\frac{F(\kappa t)}{F(t)}=\kappa^{\alpha},
\]
for all $\kappa\geq 1$. For example, each of the functions 
$t \mapsto t^{\alpha}\,$,  $\;t^{\alpha}\,(\log t)^{\beta}\,$, 
$\;t^{\alpha}\,(\log t)^{\beta}(\log\log t)^{\gamma}$ varies regularly of index $\alpha$,
whereas $t \mapsto (2+\sin t)\,t^{\alpha}$ does not
vary regularly. See \cite{BinGolTeu}.

\begin{proposition}
\label{regular variation} None of the functions $1/\mathfrak{p}(t),$ $1/N(1/t),$
$V_{\ast}(t)$ varies regularly.
\end{proposition}

\begin{proof}
By Karamata's theory, the functions $1/\mathfrak{p}(t)$ and 
$1/N(1/t)$ vary regularly simultaneously. Since $V_{\ast}(t)=1/N(1/t)$, this is true 
also for the functions
$1/\mathfrak{p}(t)$ and $V_{\ast}(t).$ 

Now assume by contradiction
that the function $V_{\ast}(t)$ is regularly varying of index $\alpha$,
which evidently must  be $\geq0$, because $V_{\ast}(t)$ is increasing. 
With our notation $s_{k}=1/\lambda_k\,$, 
choose $\varepsilon>0$ and set $a=s_k-\varepsilon$ and $b=s_k+\varepsilon$. Since 
$V_{\ast}(s_{k})=V(k),$ we have $V_{\ast}(a)\leq V(k-1)$ and $V_{\ast}(b)\geq
V(k)$. As $s_{k}\rightarrow\infty$ and $\varepsilon$ is fixed,
$s_{k}+\varepsilon<(1+\varepsilon)(s_{k}-\varepsilon)$ for large enough $k$,
whence as $k \to \infty$,
$$
2   \leq n_{k}=\frac{V(k)}{V(k-1)} \leq \frac{V_{\ast}(b)}{V_{\ast}(a)}
\leq\frac{V_{\ast}\bigl((1+\varepsilon)(s_{k}-\varepsilon)\bigr)}
{V_{\ast}(s_{k}-\varepsilon)}\rightarrow(1+\varepsilon)^{\alpha}.
$$
If $\varepsilon$ is chosen small enough, this yields the 
contradiction we were looking for. \hspace*{.5cm}
\end{proof}

\section{The operator of fractional derivative $\mathfrak{D}^{\alpha}$}\label{sec:derivative}

\setcounter{equation}{0}

Let us for a moment return to the general setting of a not necessarily
homogeneous  ultrametric measure space $(X,d,m)$ and an associated
hierarchical Laplacian $L_C$ of which we may assume without loss of
generality that it is the standard Laplacian according to \eqref{standard1} 
and \eqref{standard2}. Otherwise, we just replace $d$ by the resulting instrinsic 
ultrametric \eqref{eq:intrinsic}.

Now take  $\alpha > 0$ to introduce the new choice function 
\begin{equation}\label{C-B-fractur}
C_{\alpha}(B)=\left(  \frac{1}{\mathsf{diam}(B)}\right)^{\alpha}-
\left(  \frac{1}{\mathsf{diam}(B^{\prime})}\right)^{\alpha}, \quad B \in \mathcal{B}.
\end{equation}
We denote the resulting operator by $L_C^{^{\scriptstyle \alpha}}$.
The corresponding eigenvalues are 
$$
\lambda^{\alpha}(B)=\left(  \frac{1}{\mathsf{diam}(B)} \right)
^{\alpha},
$$
and the associated intrinsic ultrametric is $d^{\alpha}$. 

The space $\mathcal{F}$ is the linear span of all $\mathbf{1}_{B}\,$, $B \in \mathcal{B}$.
We can expand 
$$
\frac{\mathbf{1}_{B}}{m(B)} = \sum_{D \in \mathcal{B}: D \supseteq B} f_D\,,
$$
a pointwise and $\mathcal{L}^2(X,m)$-convergent series. Thus 
$$
L_C^{\alpha}\mathbf{1}_{B} = m(B) \sum_{D \in \mathcal{B}: D \supseteq B} 
\lambda^{\alpha}(D) f_D\,.
$$
From this, we can infer the following.

\begin{lemma}\label{lem:alphabeta}
For any $\alpha, \beta > 0$,
$$
L_C^{\beta}:\mathcal{F} \to \mathsf{Dom}(L_C^{\alpha}) \quad\text{with}\quad 
L_C^{\alpha}\circ L_C^{\beta}=L_C^{\alpha+\beta} \quad \text{and} \quad
(L_C^{\alpha})^{\beta}=L_C^{\alpha\beta}.
$$
\end{lemma}

Now let us briefly recall the construction of $\mathbb{Q}_p$ for arbitrary $p \ge 2$:
any non-zero element has the form
$$
x= \sum_{n=-\infty}^{\infty} a_n\,p^n \,, \;
a_n \in \{ 0, \dots, p-1\}\,, \; \exists  k: a_k \ne 0 
\;\text{and}\; a_n = 0\;\forall\ n < k.
$$
For such $x$, its $p$-adic (pseudo)norm is $\|x\|_p = p^{-k}$. The element $0$ is represented
by the series with all $a_n=0$, and $\|x\|_p =0$. In $\mathbb{Q}_p\,$, we have addition
with carries, so that it extends the analogous operation for finite sums (i.e., when 
only finititely many $a_n$ are $\ne 0$). We also have multiplication in the same
sense, so that we get a ring with unit $1$ (where $a_0=1$ and $a_n=0$ for $n \ne 0$).
The standard $p$-adic ultrametric is $\| x-y\|_p\,$. 

We also have $\| x \cdot y \|_p \le \|x\|_p\|y\|_p\,$, with equality when $p$ is prime,
and in this case, $\mathbb{Q}_p$ is of course a field.
The ring of \emph{$p$-adic integers} is 
$$
\mathbb{Z}_p = \{ x \in \mathbb{Q}_p :  \|x\|_p \le 1 \}\,.
$$
The set of non-zero distances is $\{ r_k = p^k : k \in \mathbb{Z}\}$, and the closed
ball of radius $p^k$ around $x$ is
$$
B(x,p^k) = x + p^{-k}\mathbb{Z}_p\,.
$$
Our reference measure $m$ is Haar measure of the totally disconnected, locally compact
abelian group $(\mathbb{Q}_p\,,\,+)$, nomalised such that $m(\mathbb{Z}_p) = 1$. Thus,
$$
V(k)= m\bigl(B(x,p^k)\bigr) = p^k\,.
$$
The degree sequence $(n_k)$ is constant, $n_k \equiv p$. Good references on general $p$-adic
analysis (for prime $p$) are Katok \cite{Katok} or Koblitz \cite{Koblitz}. 

\smallskip

We now take the metric $d(x,y)=\|x-y\|_p/p$ and write $\mathfrak{D} = \mathfrak{D}_p$ 
for the standard hierarchical Laplacian associated with $(\mathbb{Q}_p, d, m)$. 
Then we consider the homogeneous operator $\mathfrak{D}^{\alpha}=\mathfrak{D}^{\alpha}_p$, 
$\alpha > 0$, according to
the construction outlined at the beginning of this section. This is the
operator of \emph{$p$-adic fractional derivative} of order $\alpha$ in the terminology
of \cite{Vladimirov}, \cite{VladimirovVolovich}, \cite{Vladimirov94}. 
For $f\in\mathcal{F}$ it can be written in the form
$$
\mathfrak{D}^{\alpha}f(x)=\frac{p^{\alpha}-1}{1-p^{-\alpha-1}}
\int\limits_{\mathbb{Q}_{p}}
\frac{f(x)-f(y)}{\| x-y\|_{p}^{\alpha+1}}\,dm(y).
$$
In terms of the Fourier transform,
$$
\widehat{\mathfrak{D}^{\alpha}f}(\xi)=\| \xi\|_{p}^{\alpha}\,
\widehat{f}(\xi)\,,\quad \xi\in\mathbb{Q}_{p}\,,
$$
which means that $\mathfrak{D}^{\alpha}$ is the spectral multiplier of 
order $\alpha\,$; compare with \cite[\S 5.1]{BGPW} in the present setting. 
(This property also explains why the standard $p$-adic
metric is divided by $p$ here). In that form, it was first introduced
and studied in \cite{Taibleson75} on $\mathbb{Q}_{p}^d$ for arbitrary
$d \ge 1$. For this reason, it is sometimes also called the \emph{Taibleson Laplacian.}

The eigenvalues of $\mathfrak{D}^{\alpha}$ have the form
\begin{equation}
\lambda_k= \left(  \frac{c}{V(k)}\right)^{\alpha} =  \left(  \frac{c}{p^k}\right)^{\alpha}\,,
 \label{lambda=V}%
\end{equation}
with $c =p$. The value $c=1$ is equally reasonable, and for our computations,
the choice is not significant. 
The associated on-diagonal heat kernel with generic $c > 0$ is denoted  
$\mathfrak{p}_{\alpha}^{(c)}(t)$.
In the special case of $\mathfrak{D}^{\alpha}$, we just write
$\mathfrak{p}_{\alpha}(t)$.

\begin{theorem} \label{thm:p-adic}
There exists a non-constant strictly positive
continuous $\alpha$-periodic function $\mathcal{A}(\tau) = \mathcal{A}_p(\tau)$ 
such that on $\mathbb{Q}_p\,$,
$$%\begin{equation}
\mathfrak{p}_{\alpha}(t)=t^{-1/\alpha}\mathcal{A}(\log_{p}t).
%\text{ } \label{\mathfrak{p}(t)-formula}%
$$%\end{equation}
For fixed $\alpha > 0$, as $p$ tends to infinity,
$$
\min\mathcal{A}_{p}\sim p^{-1}(\log p)^{1/\alpha} \quad\text{and}\quad 
\max\mathcal{A}_{p}\to (e\alpha)^{-1/\alpha}\,.
$$
For general $c$ in \eqref{lambda=V}, the function $\mathcal{A}(\tau)$ has to be replaced
by  
$$
(p/c)\,\mathcal{A}\bigl(\tau+ \alpha \log_p(c/p)\bigr)\,.
$$
\end{theorem}

\begin{proof}
We give the proof for general $c>0$ in  \eqref{lambda=V}. 
By  \eqref{eq:ondiag-HK},
$$
\begin{aligned}
 \mathfrak{p}_{\alpha}^{(c)}(t)&= \sum_{k \in \mathbb{Z}}
\exp\Bigl(-t(cp^{-k})^{\alpha}\Bigr)\, \bigl(p^{-(k-1)} - p^{-k}\bigr) \\
&= (p-1)\sum_{k \in \mathbb{Z}} \exp\Bigl(-c^{\alpha} p^{\log_p t -\alpha k}\Bigr) 
\, p^{(\log_p t -\alpha k)/{\alpha}}\, p^{-\log_p t/{\alpha}}\\ 
&= t^{-1/\alpha} \,\mathcal{A}^{(c)}(\log_p t)\,,
\end{aligned}
$$
where
$$
\mathcal{A}^{(c)}(\tau) = \frac{p-1}{c} \sum_{k \in \mathbb{Z}}
\exp\Bigl(-p^{\textstyle \tau-\alpha k + \alpha \log_p c}\Bigr) 
\, p^{\textstyle (\tau -\alpha k+ \alpha \log_p c)/{\alpha}}\,.
$$
Setting $\mathcal{A}(\tau) = \mathcal{A}^{(p)}(\tau) = \frac{1}{p}\mathcal{A}^{(1)}(\tau)$
yields the proposed $\alpha$-periodic continuous function and the transformation
formula for generic $c$. It must be non-constant, since we know from
Proposition \ref{regular variation} that $t\mapsto \mathfrak{p}_{\alpha}(t)$ does not vary
regularly at $\infty$.

\smallskip

%As before, $\alpha$ is fixed. 
Deriving the asymptotics of $a_p = \min \mathcal{A}$ and 
$A_p = \max \mathcal{A}$ requires some laborious analysis. We simplify by 
%setting $q=p^\alpha$ and 
performing the change of variables $r = p^{\tau}$, and define the two functions
$$
f(r)=r^{1/\alpha}e^{-r}\quad\text{and}\quad
g(r)=\sum_{k\in\mathbb{Z}} f(r\,p^{\alpha k})\,, \quad r \ge 0.
$$
Then $\mathcal{A}(\tau) = \frac{p-1}{p}\,g(r)$, 
so that $a_p = \frac{p-1}{p}\min g$ and $A_p = \frac{p-1}{p}\max g$, both taken over the
interval $[1\,,\,p^{\alpha}]$. 
The function $f(r)$ takes its strict maximal value $f(1/\alpha)=(e\alpha)^{-1/\alpha}$ 
at the unique stationary point $r=1/\alpha$, and $f(r) \to 0$ as $r \to 0$, resp. $r \to \infty$.

We first claim that the dominant contribution to $\min$ and $\max$ comes from the 
central terms of the bi-infinite sum, that is,
\begin{equation}
g(r)=f(r) + f(r/p^{\alpha}) + O(1/p)\,,\quad\text{as }\; p \to \infty
\label{G-claim}%
\end{equation}
uniformly for $r \in [1\,,\,p^{\alpha}]$. Indeed, we write%
\[
g(r)=f(r) + f(r/p^{\alpha})
+\textrm{Sum}_I+ \textrm{Sum}_{I\!I}\,,
\]
where%
\[
\textrm{Sum}_I = \sum_{k\geq1}f(r\,p^{\alpha k})\quad\text{and}\quad
\textrm{Sum}_{I\!I} = \sum_{k\geq2} f(r\,p^{-\alpha k}).
\]
When $p$ is sufficiently large then for all $k\geq1$ and all $r \in [1\,,\,\infty)$%
$$
f(r\,p^{\alpha k}) \le f(p^{\alpha k}) = p^k \exp(- p^{\alpha k}) \le p^{-k}\,,
$$
because $f(r)$ is decreasing beyond $1/\alpha$.
It follows that $\textrm{Sum}_I < 2/p\,$.
Similarly, when $p$ is sufficiently large, 
for all $k\geq2$ and all $r \in [1\,,\,p^{\alpha}]$
$$
f(r\,p^{-\alpha k})  \le f(p^{-\alpha (k-1)}) < p^{-(k-1)}\,.
$$
Therefore also $\textrm{Sum}_{I\!I} <2/p\,$.
Thus, we are lead to study the function
$$
h :[1\,,\,p^{\alpha}]\to \mathbb{R}_+\,,\quad
h(r) = f(r) + f(r/p^{\alpha})
$$
dependig on $p$ and $\alpha$. 
By \eqref{G-claim}, 
\[
a_p=\min_{[1\,,\,p^{\alpha}]} \,h +O(1/p) \quad\text{and}\quad
A_p =\max_{[1\,,\,p^{\alpha}]} \,h + O(1/p)\,,\quad\text{as }\;p\rightarrow\infty\,.
\]
We claim that as $p\rightarrow\infty,$%
\[
\min_{[1\,,\,p^{\alpha}]} h \sim\frac{(\log p)^{1/\alpha}}{p}\quad\text{and}\quad
\max_{[1\,,\,p^{\alpha}]} h\to (e\alpha)^{-1/\alpha}\,,
\]
and this will complete the proof of the theorem. To prove this, we distinguish three
cases.
\\[5pt]
\emph{Case 1.} $\; 0 < \alpha < 1$. Looking for the stationary points of $h(r)$,
we transform the equation $h'(r) = 0$ into 
\begin{equation}\label{eq:stat1}
\frac{r - \alpha^{-1}}{\alpha^{-1} - r\,p^{-\alpha}} =  
\frac{1}{p}\, \exp\bigl(r(1-p^{-\alpha})\bigr).
\end{equation}
Write $u(r)$ for the left hand side, and $v(r)$ for the right hand side. The denominator
of $u(r)$ does not vanish in $[1\,,\,p^{\alpha}]$. Both functions are strictly increasing
and strictly convex in that interval. Within the interval $[1\,,\,p^{\alpha}]$, for large $p$
we find two solutions $r_p < s_p$ of \eqref{eq:stat1}. Namely, 
one sees that for any $\varepsilon > 0$, if $p$ is large enough,
$$
\begin{aligned}
&\begin{cases} 0 = u(\alpha^{-1}) < v(\alpha^{-1})\,,&\\
              u(\alpha^{-1} + \varepsilon) > v(\alpha^{-1} + \varepsilon)\,,&
\end{cases}
\\
&\begin{cases}
u(\log p + \log\log p^{\alpha-\varepsilon}) >  v(\log p + \log\log p^{\alpha-\varepsilon})\,,&\\
u(\log p + \log\log p^{\alpha+\varepsilon}) <  v(\log p + \log\log p^{\alpha+\varepsilon})\,.&
\end{cases}
\end{aligned}
$$
Thus, we get $\alpha^{-1} < r_p < \alpha^{-1} + \varepsilon$ and 
$\log p + \log\log p^{\alpha-\varepsilon} < s_p < \log p + \log\log p^{\alpha+\varepsilon}$,
whence
$$
r_p \to  \alpha^{-1} \quad \text{and} \quad s_p \sim s_p^* = \log p + \log\log p^{\alpha}\,,
\quad \text{as }\; p \to \infty\,.
$$
Let us show that there are no further solutions of \eqref{eq:stat1} in
$[1\,,\,p^{\alpha}]$: We compute $v' - u'$ and find that in our interval 
\begin{equation}\label{eq:sqrt}
\begin{aligned}
&\bigl(u'(r)\bigr)^{-1/2} - \bigl(v'(r)\bigr)^{-1/2}\\
&\quad= 
(1-p^{-\alpha})^{1/2}\left(\bigl(\alpha^{-1/2}- r\,\alpha^{1/2}\, p^{-\alpha}\bigr)
- p^{1/2} \,\exp\Bigl(-r \,\frac{1-p^{-\alpha}}{2}\Bigr)\right).
\end{aligned}
\end{equation}
For large $p$, the right hand side is $<0$ at $r=0$ and $> 0$ at $r = p^{\alpha}$. 
By strict convexity of the exponential function, there must be precisely one root
of $v' - u'$ in $[1\,,\,p^{\alpha}]$, and it must be located
between $r_p$ and $s_p$. Now, if \eqref{eq:stat1} had more than two solutions, then
$v' - u'$ would have more than one root, which is not the case.

Tracing back the comparison between $u(r)$ and $v(r)$ to the sign of $h'(r)$, we see
that the maximum of $h$ in $[1\,,p^{\alpha}]$ is attained at $r_p\,$, and the minimum at
$s_p\,$. 
Now a short asymptotic computation yields
$$
\begin{aligned}
\max_{[1\,,\,p^{\alpha}]} h &= h(r_p) \sim h(\alpha^{-1}) 
\sim f(\alpha^{-1}) = (e\alpha)^{-1/\alpha}\,,
\quad \text{and}\\
\min_{[1\,,\,p^{\alpha}]} h &= h(s_p) \sim h(s_p^*) 
\sim f(s_p^* / p^{\alpha} ) \sim\frac{(\log p)^{1/\alpha}}{p}\,,
\end{aligned}
$$
as proposed.
\\[5pt]
\emph{Case 2.} $\; \alpha = 1$. In this case, 
we find an additional solution $t_p > s_p$ of \eqref{eq:stat1} in $[1\,,\,p)$.
Namely, for any $\varepsilon > 0$, if $p$ is sufficiently large then
$$
\begin{cases} u\bigl((1-\varepsilon)p\bigr) < v((1-\varepsilon)p\bigr)\,,&\\
              \infty = u(p-) > v(p)\,.&
\end{cases}
$$
Now note that there can be no further roots of \eqref{eq:stat1} in $[1\,,\,p]$,
since otherwise $v^{\prime} - u^{\prime}$ would have more than two roots,
which is impossible in view of \eqref{eq:sqrt}. 

We see that at $r_p$ and $t_p$ we have relative maxima of $h$, with $r_p \to 1$ and 
$t_p \sim p$ as $p \to \infty$. For large $p$, we have  $h(r_p) > h(t_p)$,
so that the absolute maximum is at $r_p$, and $h(r_p) \to e^{-1}$, as proposed.
The minimimum at $s_p$ and the asymptotics for $h(s_p)$ are as in Case~1. 
\\[5pt]
\emph{Case 3.} $\; \alpha > 1$. In this case, the minimum at $s_p$ and the 
corresponding asymptotics
remain unchanged. To locate $r_p$ (which is $> s_p$)
it is better to invert  both sides
of \eqref{eq:stat1} to get
\begin{equation}\label{eq:stat2}
\frac{\alpha^{-1} - r\,p^{-\alpha}}{r - \alpha^{-1}} =  
p\, \exp\bigl(-r(1-p^{-\alpha})\bigr).
\end{equation}
For any $\varepsilon \in (0\,,\,\alpha^{-1})$, if $p$ is large enough then
$$
\begin{cases} 0 = 1/u(\alpha^{-1} p^{\alpha}) < 1/v (\alpha^{-1} p^{\alpha})\,,&\\
 1\big/u\bigl((\alpha^{-1}-\varepsilon) p^{\alpha}\bigr) > 
 1\big/\bigl((\alpha^{-1}-\varepsilon) p^{\alpha}\bigr).&\\
\end{cases}
$$
Thus $(\alpha^{-1}-\varepsilon) p^{\alpha} < r_p < \alpha^{-1} p^{\alpha}$, and 
$r_p \sim  \alpha^{-1} p^{\alpha}$. The argument to show that there are no
solutions of \eqref{eq:stat2} besides $r_p$ and $s_p$ is as in Case 1.  
We get once more
$$
\max_{[1\,,\,p^{\alpha}]} h = h(r_p) \sim h(\alpha^{-1}p^{\alpha}) \sim f(\alpha^{-1}) 
= (e\alpha)^{-1/\alpha}\,,
$$
which concludes the proof.
\end{proof}

We shall now generalise Theorem \ref{thm:p-adic} in two directions:
first, we replace equality in \eqref{lambda=V} by asymptotic equivalence.
Second, we illustrate how the asymptotics of $\mathfrak{p}(t)$ depends only
on the behaviour of the choice function (or equivalently, the eigenvalue function)
for large balls, when $t \to \infty$, resp. for small balls, when $t \to 0$.

\begin{theorem}\label{thm:alphabeta} Let $L_C$ be a homogeneous Laplacian with 
$$
V(k)=p^k\quad\text{and}\quad 
\lambda_k \sim \begin{cases} (c_+/p^k)^{\alpha}\,,\;&\text{as }\; k \to +\infty\,,\\
                             (c_-/p^k)^{\beta}\,,\;&\text{as }\; k \to -\infty\,,
               \end{cases}
\qquad \alpha, \beta > 0\,.
$$
Then there are constants $A_+ > a_+ > 0$ and $A_- > a_- > 0$ depending on $p$ such that
$$
\begin{aligned}
\limsup_{t \to \infty} t^{1/\alpha}\,\mathfrak{p}(t) = A_+ \quad&\text{and}\quad
\liminf_{t \to \infty} t^{1/\alpha}\,\mathfrak{p}(t) = a_+\,, \quad\text{while}\\
\limsup_{t \to 0} t^{1/\beta}\,\mathfrak{p}(t) = A_- \quad&\text{and}\quad
\liminf_{t \to 0} t^{1/\beta}\,\mathfrak{p}(t) = a_-\,.
\end{aligned}
$$
Moreover, as $p\to \infty$, we have 
$$
\begin{aligned}
A_+ \sim p\, c_+^{-1} (e\alpha)^{-1/\alpha} \quad&\text{and}\quad 
a_+ \sim c_+^{-1}(\log p)^{1/\alpha}\,,\quad\text{while}\\
A_- \sim p\, c_-^{-1} (e\beta)^{-1/\beta} \quad&\text{and}\quad 
a_- \sim c_-^{-1}(\log p)^{1/\beta}\,.
\end{aligned}
$$
\end{theorem}

\begin{proof}
We start with $t\to \infty$. Let $\lambda^+_k = (c_+/p^k)^{\alpha}$.
Take $\varepsilon >0$ and choose $N=N_{\varepsilon}$ such that
\begin{equation}\label{eq:eps}
(1-\varepsilon)\lambda^+_k \le \lambda_k \le (1+\varepsilon)\lambda^+_k
\end{equation}
for all $k > N$. We decompose
$$
\mathfrak{p}(t) = (p-1)\left( \sum_{k \le N} + \sum_{k > N} \right) e^{-t\lambda_k} p^{-k}\,. 
$$
We have $\lambda_k \ge \lambda_N$ for $k \le N$, so that for $t \to \infty$,
$$
\begin{aligned}
t^{1/\alpha}(p-1)\sum_{k \le N} e^{-t\lambda_k} p^{-k} 
&\le t^{1/\alpha}e^{-t\lambda_N/2}(p-1)\sum_{k \le N} e^{-t\lambda_k/2} p^{-k}\\
&\le t^{1/\alpha}e^{-t\lambda_N/2} \mathfrak{p}(t/2) \le t^{1/\alpha}e^{-t\lambda_N/2} \mathfrak{p}(1/2) 
\to 0.
\end{aligned}
$$
Writing $\mathfrak{p}_{\alpha}^{(c)}(t)$ for the heat kernel associated with the Laplacian of
\eqref{lambda=V}, we also have that
$$
t^{1/\alpha}(p-1)\sum_{k \le N} e^{-t\lambda_k^+} p^{-k} 
\le t^{1/\alpha}e^{-t\lambda_N^+/2}\,\, \mathfrak{p}_{\alpha}^{(c_+)}(1/2) \to 0\,,\quad\text{as }\;
t \to \infty\,.
$$
Now
$$
\sum_{k > N} e^{-t(1+\varepsilon)\lambda_k^+} p^{-k} 
\le \sum_{k > N}  e^{-t\lambda_k} p^{-k}
\le \sum_{k > N} e^{-t(1-\varepsilon)\lambda_k^+} p^{-k} .
$$
Thus, the preceding two estimates yield
$$
\begin{aligned}
\bigl(t(1+\varepsilon)\bigr)^{1/\alpha}\,&\mathfrak{p}_{\alpha}^{(c_+)}\bigl(t(1+\varepsilon)\bigr)
\;+\; o(1) \le \;t^{1/\alpha}\mathfrak{p}(t)\\
&\le \bigl(t(1-\varepsilon)\bigr)^{1/\alpha}\,\mathfrak{p}_{\alpha}^{(c_+)}\bigl(t(1-\varepsilon)\bigr)
\;+\; o(1)\,, \quad\text{as }\; t \to \infty.
\end{aligned}
$$
This and Theorem \ref{thm:p-adic} yield the proposed asymptotic behaviour at infinity.

\smallskip

To get the proposed asmptotic estimate when $t \to 0$, we now choose
$M=M_{\varepsilon}$ such that \eqref{eq:eps} holds for all $k \le M$.
In this case,  
$$
t^{1/\beta}(p-1)\sum_{k > M} e^{-t\lambda_k} p^{-k} \le t^{1/\beta} p^{-M} \to 0\,,
\quad\text{as }\;
t \to 0\,.
$$ 
The same holds with $\lambda_k^-$ in the place of $\lambda_k\,$.
This time, we have 
$$
\sum_{k \le M} e^{-t(1+\varepsilon)\lambda_k^-} p^{-k} 
\le \sum_{k \le M}  e^{-t\lambda_k} p^{-k}
\le \sum_{k \le M} e^{-t(1-\varepsilon)\lambda_k^-} p^{-k}.
$$
We are lead to 
$$
\begin{aligned}
\bigl(t(1+\varepsilon)\bigr)^{1/\beta}\,&\mathfrak{p}_{\beta}^{(c_-)}\bigl(t(1+\varepsilon)\bigr)
\;+\; o(1) \le \;t^{1/\beta}\mathfrak{p}(t)\\
&\le \bigl(t(1-\varepsilon)\bigr)^{1/\beta}\,\mathfrak{p}_{\beta}^{(c_-)}\bigl(t(1-\varepsilon)\bigr)
\;+\; o(1)\,, \quad\text{as }\; t \to 0.
\end{aligned}
$$
This completes the proof.
\end{proof}

\begin{example}\label{exa:sum}\rm
Theorem \ref{thm:alphabeta} applies, in particular, to the homogeneous hier\-archical Laplacian 
$\mathfrak{D^{\alpha}} +  \mathfrak{D^{\beta}}$ on $\mathbb{Q}_p\,$, where $0 < \alpha < \beta$. 
In this case, $c_+ = c_- = p$. \hfill\rule{0.5em}{0.5em}
\end{example}

\section{Random walks on ultrametric groups}\label{sec:rw}
\setcounter{equation}{0}

We now recall the construction of our processes in terms of their semigroup $(P^t)_{t > 0}$
of transition operators, as introduced and studied in \cite{BGP} and \cite{BGPW}.
On our non-compact ultrametric measure space $(X,d,m)$ with the collection $\mathcal{B}$
of balls with positive measure, we have the famliy of averaging operators
$$
Q_Bf(x) = \frac{1}{m(B)}\int_B f\,dm\,
$$
which appear in the definition \eqref{eq:hlaplacian} of the hierarchical Laplacian.
We take a probability measure $\sigma$ supported by all of $\mathbb{R}_+$ and 
write $\sigma^t(r) = \sigma\bigl([0\,,\,r)\bigr)^t$. With respect to the given 
metric $d$, 
$$
P^tf(x) = \int_{\mathbb{R}_+} Q_{B(x,r)}f(x)\, d\sigma^t(r)\,.
$$
Then the infinitesimal generator of this Markov semigroup is the hierarchical
Laplacian $L_C\,$, where the choice function $C(\cdot)$ is given via \eqref{eq:C1condition}
by the eigenvalue function: 
$$
C(B) = \lambda(B) - \lambda(B') \quad\text{with}\quad 
\lambda(B) = \log \bigl(1/\sigma(\textsf{diam}\, B)\bigr)\,,\quad B \in \mathcal{B}.
$$ 
The transition operator $P = P^1$ gives rise to a discrete-time Markov chain 
(``isotropic random walk'') $(\mathcal{X}_n)_{n \ge 0}$  on $X$ which is embedded
in the continuous-time Markov process $(\mathcal{X}_t)$ whose infinitesimal
generator is $L_C\,$. The transition rule in one step is as follows: we first choose
a random step length $r$ according to the probability measure $\sigma$ and
then move to a point chosen according to $m$-uniform distribution within the (closed)
ball of radius $r$ around the current position in $X$.

Now let us consider the homogeneous situation. As noticed in \cite{Figa-Tal1}, 
\cite{Figa-Tal2}, the measure space $(X,m)$ can then be identified with a locally compact
totally disconnected group $\mathfrak{G}$ equipped with its Haar measure. (Indeed, we may even identify
it with an Abelian group.) Our homogeneous Laplacian $L_C$  and 
the transition operator $P$ are group-invariant, and $(\mathcal{X}_n)$ is a random
walk on that group. Write $e$ for the group identity ($0$ in the Abelian case),
and let $(r_k)$ be the sequence of distances of \eqref{eq:cases}. Then
$
\mathfrak{G}_k = B(e,r_k)
$
is a compact-open subgroup of $\mathfrak{G} \equiv X$,
$$
\mathfrak{G} = \bigcup_k \mathfrak{G}_k\,,\quad \text{and}\quad n_k = [\mathfrak{G}_k : \mathfrak{G}_{k-1}]
$$
gives the degree sequence. The collection $\mathcal{B}_k$ of balls with radius $r_k$
consists of the left cosets of $\mathfrak{G}_k$ in $\mathfrak{G}$. 
We usually normalise the Haar measure $m$ such that $m(\mathfrak{G}_0)=1$. 
In the discrete Case 2 of \eqref{eq:cases}, $\mathfrak{G}_0=\{e\}$ and $m$ 
is the counting measure. The random walk (as well as
the continuous-time process) is induced by the probability measure with density
\begin{equation}\label{eq:mu}
\frac{d\mu}{dm} = \sum_k  \frac{\pi_k}{m(\mathfrak{G}_k)}\, \boldsymbol{1}_{\mathfrak{G}_k}\,, 
\quad \text{where}\quad \pi_k = \sigma(r_{k+1}) - \sigma(r_k).
\end{equation}
Here, $k$ ranges over $\mathbb{Z}$ in Case 1 and over $\mathbb{N}_0$ in Case 2
of \eqref{eq:cases}.
In this situation, for the step length distribution $\sigma$ it is enough
to specify the values~$\pi_k\,$. 

\begin{example}\label{ex-padic}\rm 
In the case of $\mathbb{Q}_p\,$ and the operator $\mathfrak{D}^{\alpha}$, 
we have $\mathfrak{G}_k = p^{-k}\mathbb{Z}_p\,$, and with respect to the standard
$p$-adic metric and Haar measure $m$ with $m(\mathbb{Z}_p)=1$, 
we have $r_k=p^k$ and
$$
\pi_k = \exp(-p^{-\alpha k}) - \exp(-p^{-\alpha (k-1)}) \sim 
\begin{cases} (p^{\alpha}-1)p^{-\alpha k}\,,\;&\text{as }\; k \to + \infty\,,\\
         \exp(-p^{-\alpha k}) \,,\;&\text{as }\; k \to -\infty\,.
\end{cases}
$$
The corresponding transition kernel asymptotics are then provided by Theorem
\ref{thm:p-adic}, regardless of whether time is disrete or continuous.

\smallskip
 
The same also applies more generally to the Taibleson Laplacian with 
parameter $\alpha$ on $\mathbb{Q}_p^d\,$, which has the same degree and eigenvalue
sequences as $\mathfrak{D}^{\alpha/d}$ on $\mathbb{Q}_{p^d}\,$, compare with 
\cite[Subsection 5.3.2]{BGPW}.

\smallskip

Indeed, the homogenous hierarchical Laplacians do not ``see'' the inner algebraic
structure of each $\mathfrak{G}_k\,$. Thus, instead of $\mathbb{Q}_p^d$
we might also take $X = (\mathcal{R}_p[t])^d$,
where $\mathcal{R}_p[t]$ is the ring of formal Laurent series in the variable $t$
over the ring $\mathcal{R}_p = \mathbb{Z}/(p\mathbb{Z})$ of integers 
modulo $p$.\hfill\rule{0.5em}{0.5em}
\end{example}

Next, let us consider the situation where $\mathfrak{G}$ is a discrete group, so that
we are in Case 2 of \eqref{eq:cases}. In this case, $\mathfrak{G}_0 = \{ e \}$,
and $(\mathfrak{G}_k)_{k \ge 0}$ is a strictly increasing sequence of finite
subgroups of $\mathfrak{G}$. In other words, $\mathfrak{G}$ is a countable, \emph{locally finite}
group (every finititely generated subgroup is finite). In this situation,
$\mathfrak{p}(t)$ is the probability that $(\mathcal{X}_t)$ starting at the unit element 
is at $e$ at time $t$. For integer time, $\mathfrak{p}(n) = \mu^{(n)}(e)$ is the
$n$-step return probability to the starting point, and $\mu^{(n)}$ is the $n$-th 
convolution power of $\mu$.

The following is practically immediate from Theorem \ref{thm:alphabeta}.

\begin{theorem}\label{thm:alpha}
If $X = \mathfrak{G} = \bigcup_{k \ge 0} \mathfrak{G}_k$ as above, and the degree sequence $(n_k)_{k \ge 1}$
and the step length weights $(\pi_k)_{k \ge 0}$ satisfy
$$
%\begin{aligned}
%[\mathfrak{G}_k : \mathfrak{G}_{k-1}]
n_k = [\mathfrak{G}_k : \mathfrak{G}_{k-1}] 
= p \ge 2 \quad \text{and}\quad
\pi_k \sim \tilde c \, p^{-\alpha k} \; \text{ as }\; k \to \infty\,, 
%\end{aligned}
$$
where $\tilde c, \alpha > 0$,
then there are constants $A > a > 0$ such that
$$
\limsup%_{t \to \infty} 
t^{1/\alpha}\,\mathfrak{p}(t) = A \quad\text{and}\quad
\liminf%_{t \to \infty} 
t^{1/\alpha}\,\mathfrak{p}(t) = a\,,
$$
as $t \to \infty$ in $\mathbb{R}_+$ or in $\mathbb{N}$.
Moreover, as $p\to \infty$, we have 
$$
A \sim p\, \bigl(\tilde c \, e\, \alpha\bigr)^{-1/\alpha} \quad\text{and}\quad 
a \sim \bigl(\tilde c/\log p\bigr)^{-1/\alpha}\,.
$$
\end{theorem}

\begin{proof}
We know that $\sigma_k = \sigma(r_k) \to 1$ as $k \to \infty$. Therefore
$$
\lambda_k = - \log \sigma_k = - \log \Bigl(1 -  \sum_{j \ge k} \pi_j \Bigr)
\sim  \sum_{j \ge k} \pi_j \sim \tilde c(1-p^{-\alpha}) \, p^{-\alpha k} = (c/p^k)^{\alpha}\,,
$$
where $c = \bigl(\tilde c(1-p^{-\alpha})\bigr)^{1/\alpha}$.
By \eqref{eq:ondiag-HK},
$$
\mathfrak{p}(t) = (p-1) \sum_{k \ge 1} e^{-t\lambda_k} p^{-k}\,.
$$ 
We are now in precisely the same situation as in the first part of the 
proof of Theorem \ref{thm:alphabeta}, which leads to the desired result.
\end{proof}

\begin{remark}\label{rmk:Cartwright}\rm
In this context the groups studied in \cite{Cartwright} comprise 
$$
\mathfrak{G} = \bigoplus_{j=1}^{\infty} \mathbb{Z}(p)_j\,,
$$
the direct sum of countably many copies $\mathbb{Z}(p)_j$ of the additive group 
$\mathbb{Z}(p) = \{ 0, 1,\dots, p-1\}$ with addition modulo $p$.
Thus, setting
$$
\mathfrak{G}_k = \{ (x_1\,,x_2\,,\dots) \in \mathbb{Z}(p)^{\mathbb{N}} : 
           x_j = 0\; \text{for all} \; j > k \}\,,
$$
we have $\mathfrak{G} = \bigcup_k \mathfrak{G}_k$ with the resulting ultrametric structure
described above. Each $\mathfrak{G}_k$ can be identified with the
direct product of $k$ copies of $\mathbb{Z}(p)$. 
The random walks on $\mathfrak{G}$ (and other infinite direct sums) considered 
in \cite{Cartwright} are as follows: start with a symmetric probability
measure $\mu_0$ on $\mathbb{Z}(p)$, and let $\mu_k$ be the copy of $\mu_0$
on the $j$-th coordinate in $\mathfrak{G}$, that is,
$$
\mu_k(x_1\,,x_2\,,\dots) = \begin{cases} \mu_0(x_k)\,,\;&\text{if }\; x_j = 0 \;\text{ for all }
                                                          j \ne k\,,\\      
                                         0\,,\;&\text{otherwise.}                          
                           \end{cases}
$$
Then let 
$$
\mu = \sum_k \pi_k\, \mu_k\,, \quad \text{where}\quad \pi_k \ge 0\,,\; \sum_k \pi_k=1\,.
$$
This is not the same as our probability measure of \eqref{eq:mu}, for which instead of $\mu_k$  
we have uniform distribution on $\mathfrak{G}_k\,$. Indeed, the above $\mu$ is not related with a
hierarchical Laplacian. Nevertheless, for values of $\pi_k$ as
in Theorem \ref{thm:alpha} (with equality instead of asymptotic equivalence, and stating
only the case $p=2$), Cartwright \cite[Thm. 2]{Cartwright} obtains asymptotics
of the same type as ours and observes their periodic ossilation. However, he does
not prove that the upper and lower bound in that oscillation are indeed 
distinct. \hfill\rule{0.5em}{0.5em}
\end{remark}

\section{Isotropic random walks on the infinite symmetric group}\label{sec:symmetric}
\setcounter{equation}{0}

In this final section, we study another class of cases that are part of the general setting 
of locally finite groups described in the preceding Section \ref{sec:rw}. We consider the group
$\mathfrak{G} = S_{\infty}$ of all permutations of $\mathbb{N}$ which leave all but finitely many elements fixed.
Here, $\mathfrak{G}_k = S_k\,$, the group of all permutations of $\{ 1, \dots, k\}$, here interpreted
as those permutations of $\mathbb{N}$ which fix each $n > k$. The balls in $\mathcal{B}_k$ are the
left cosets of $S_{k+1}\,$. Thus, the degree sequence and 
the resulting volume function are given by
\begin{equation}\label{eq:nk-Vk}
n_k = k+1 \quad\text{and}\quad V(k) = (k+1)!
\end{equation}
As in the previous sections, we consider a class of hierarchical Laplacians
on $S_{\infty}$ whose eigenvalues are
a function of the volume (number of elements) of the balls. Consider the function
\begin{equation}\label{eq:Lambda}
\Lambda: \mathbb{R}_+ \to \mathbb{R}_+\,,\quad \Lambda(v) = v^{-\alpha}\phi(\log v),
\end{equation}
where $\alpha > 0$ and the function $\phi(\cdot)$ varies regularly at $\infty$. 
We assume that $\Lambda(\cdot)$ is strictly decreasing and that $\phi(\cdot)$ varies smoothly,
which is no loss of generality for our asymptotic estimates 
\cite[pages 44--55]{BinGolTeu}.
The eigenvalue sequence for the associated hierarchical Laplacian, as well as 
the weight sequence for the probability measure $\mu$ generating the associated 
random walk (in discrete or continuous time) according to \eqref{eq:mu} are then
equivalently given by
\begin{equation}\label{eq:lambdapi}
\lambda_k = \Lambda\bigl( V(k) \bigr)\quad \text{and}\quad
\pi_k =  \tilde \Lambda\bigl( V(k) \bigr)
\,, %= \Lambda\bigl( \Gamma(k+1) \bigr)
\end{equation}
where $\tilde\Lambda$ is as $\Lambda$, with another smoothly varying 
function $\tilde\phi$ in the place of $\phi$. 

This should be compared with \cite{Lawler} and \cite{BrofWoess2001}. In \cite{BrofWoess2001},
a similar random walk appears in the briefly mentioned ``fifth model'', while the other
random walks considered there do not arise from a hierarchical ($\equiv$ isotropic) model.
The proabibility weights used there are of the form $\pi_k \sim c/\Gamma(1+ k-\beta)$, where
$\beta > 0$. They all correspond to the case $\alpha=1$ of
\eqref{eq:Lambda}. 

By Proposition \ref{pro:order} the function $1/\mathfrak{p}(t)$ is of finite
order $1/\alpha>0$ whereas by Proposition \ref{N-doubling} $\mathfrak{p}(t)$
is not doubling, as opposed to the function $\mathfrak{p}_{\alpha}(t)$ of Section \ref{sec:derivative}.
Before stating our asymptotic result, which extends and strengthens the one in \cite{BenLSC2016} 
substantially, we need to introduce the following positive
real functions for sufficiently large $t$.
\begin{equation}\label{eq:psiPsi}
\begin{aligned}
 \rho(t) &= \frac{\log t}{\alpha\, \log\log t}\,, \\
\psi(t) &= t^{-1/\alpha}
\left( \frac{\log \rho(t)}{\phi\bigl(\rho(t)\log \rho(t)\bigr)}\right)^{\!1/\alpha} 
\quad \text{and} \\
\Psi(t) &= t^{-1/\alpha}
\left( \frac{\rho(t)^{\alpha}}{e\,\alpha\,\phi\bigl(\rho(t)\log \rho(t)\bigr)}\right)^{\!1/\alpha} 
\end{aligned}  
\end{equation}

\begin{theorem}\label{thm:Sinfty}
For the hierarchical Laplacian, resp. isotoropic random walk on $S_{\infty}$ defined
via \eqref{eq:nk-Vk}, \eqref{eq:Lambda} and \eqref{eq:lambdapi},
$$
\limsup \frac{\mathfrak{p}(t)}{\Psi(t)}=1 \quad \text{and}\quad
\liminf \frac{\mathfrak{p}(t)}{\psi(t)}=1\,,
$$
as $t \to \infty$ in $\mathbb{R}$ or in $\mathbb{N}$, respectively.  
\end{theorem}

\begin{proof} Using the Euler gamma function, we consider
$$
F(t,r) = t\,\Lambda\bigl( \Gamma(1\!+\!r) \bigr)  + \log \Gamma(1\!+\!r) - \log r\,.
$$
We start the labourious proof by using \eqref{eq:ondiag-HK} to write
\begin{equation}\label{eq:Laplace}
\mathfrak{p}(t) = \sum_{k \ge 1} e^ {-\lambda_k t} \left( \frac{1}{V(k-1)} - \frac{1}{V(k)}\right)
\sim \sum_{k \ge M} e^{-\Lambda(k!)} \frac{k}{k!} = \sum_{k \ge M} e^{-F(t,k)}
\end{equation}
for arbitrarily chosen (fixed) $M$, as $t \to \infty\,$. 
Indeed, the initial terms of the sum are exponentially
smaller than the remainder of the series, as $t \to \infty$. We follow Laplace's 
method \cite[Par. 4.12.8]{BinGolTeu} to study the asymptotic behaviour of the rightmost series
in \eqref{eq:Laplace}. 

\smallskip
\noindent
\textbf{Step 1. Localisation of the maximum.} We look for the point $\bar r(t)$ where, for large $t$,
the function $-F$ assumes its maximum in $r$. First of all, it follows from
%it is well-known 
\cite[Formula 6.3.16]{Abram} that
%$$
%\Phi(r) = \frac{\Gamma'(1+r)}{\Gamma(1\!+\!r)}
%= -\gamma+\sum\limits_{n\geq1}\left(  \frac{1}{n}-\frac{1}{n+r}\right),
%$$
%where $\gamma$ is the Euler constant. 
%the function $\Phi$ is strictly increasing,
$$
\Phi(r) = \Gamma^{\prime}(1+r)/\Gamma(1+r) \sim\log r 
%\quad \text{and}\quad\Phi^{\prime}(r)\sim\frac{1}{r}\,,\, 
\quad \text{as }\; r \to  \infty\,, 
$$
a strictly increasing function. Abbreviating $\overline{\Lambda}(v) = v\Lambda^{\prime}(v)$,  
we compute the %second 
first partial derivative of $F$ with respect to $r$:
$$
F_{r}(t,r)  
= t\,  \Phi(r)\,\overline{\Lambda}\bigl(\Gamma(1\!+\!r)\bigr)
\,+\,\Phi(r)\,-\,1/r\,. 
$$
%$$
%F_{rr}(t,r)  
%= t\left(  \Phi^{2}(r)\,\, \Gamma(1\!+\!r)\,\,\overline{\Lambda}^{\prime}\!\bigl(\Gamma(1\!+\!r)\bigr)
%+\Phi^{\prime}(r)\,\, \overline{\Lambda}\bigl(\Gamma(1\!+\!r)\bigr)\right)
%\,+\,\Phi^{\prime}(r)\,+\,1/r^{2}. 
%$$
Since $\Lambda(\cdot)$ varies smoothly along with $\phi$, as $v \to \infty$
\begin{equation}\label{eq:ovLambda}
\overline{\Lambda}(v) = -\alpha\, \Lambda(v) \bigl(1+o(1)\bigr)
\quad \text{and}\quad 
v \, \overline{\Lambda}^{\prime}(v) = \alpha^2\, \Lambda(v) \bigl(1+o(1)\bigr)\,.
\end{equation}
Therefore, for arbitrary $t > 0$ we can choose $M$ such that $r \mapsto F_r(t,r)$
is strictly increasing on $[M\,,\infty)$. Thus, $r \mapsto F(t,r)$ is strictly convex
on that interval, and there is at most one stationary point which must 
%$$
%F_{rr}(t,r)  \sim t \, \Lambda\bigl(\Gamma(1\!+\!r)\bigr) 
%             \bigl( (\alpha\,\log r)^2 - \alpha/r \bigr) \,+\, (r+1)/r^2\,
%$$
%so that we can choose $M$ sufficiently large such that $F_{rr}(t,r)> 0$
%for all $t > 0$ and all $r  \ge M$. 
solve the equation 
\begin{equation}\label{eq:bar-r}
- t \, \overline\Lambda\bigl(\Gamma(1\!+\!r)\bigr) = 1 - \frac{1}{r\Phi(r)}. 
\end{equation}
For any fixed $t$, when $r \to \infty$ the left hand side tends to $0$,
while the right hand side tends to $1$. On the other hand, when $t$ is large enough -- say 
$t \ge t_M$ -- then at $r =M$ the left hand side is strictly bigger than the right hand side.
Therefore, there is a unique solution $\bar r(t)$ of \eqref{eq:bar-r}, and this is where the
minimum of $r \mapsto F(t,r)$ over $[M\,,\infty)$ is located.
Smoothness of the functions appearing in \eqref{eq:bar-r} implies that $t \mapsto \bar r(t)$
is smooth. Furthermore, no matter how large $M$ is chosen, the above shows that 
$\bar r(t) > M$ for $t \ge t_M\,$. Thus, $\bar r(t) \to \infty$ as $t \to \infty$.
Also, $\bar r(\cdot)$ is strictly increasing. Set 
$\bar v(t) = \Gamma\bigl((1+\bar r(t)\bigr)$. Then \eqref{eq:ovLambda} yields that as 
$t \to \infty\,$,
\begin{equation}\label{eq:alpha-t}
t\,\Lambda\bigl(\bar v(t)\bigr) = \frac{1}{\alpha} + o(1)\,,\quad \text{and}
\quad \log \frac{1}{\Lambda\bigl(\bar v(t)\bigr)} = \log(\alpha t) + o(1)\,.
\end{equation}
On the other hand, by \eqref{eq:Lambda} and Stirling's formula \cite[Formula 6.1.41]{Abram},
$$
\log \frac{1}{\Lambda\bigl(\bar v(t)\bigr)} \sim
\alpha \, \log \bar v(t) \sim \alpha\,\bar r(t) \log \bar r(t) \,,\quad
\text{as }\; t \to \infty\,.
$$ 
Now we see that $\alpha\,\bar r(t) \log \bar r(t) \sim \log(\alpha t) \sim \log t$, and 
$$
\log \bar r(t) \sim \log \Bigl(\alpha\,\bar r(t) \log \bar r(t)\Bigr) \sim \log \log t\,.
$$
We conclude that
\begin{equation}\label{eq:rbar-rho}
\bar r(t) \sim \rho(t)\quad\text{and}\quad e^{-F(t,\bar r(t))}\sim \Psi(t) \,,  
\quad \text{as }\; t \to \infty\,.
\end{equation}
The second asymptotic formula in \eqref{eq:rbar-rho} follows from the first one
together with \eqref{eq:alpha-t} and \eqref{eq:Lambda}:
$$
\begin{aligned}
e^{-F(t,\bar r(t))} &\sim e^{-1/\alpha}\,  \bar r(t)\,\bar v(t)^{-1} 
\sim e^{-1/\alpha}\, \bar r(t) \, \Lambda\bigl(\bar v(t)\bigr)^{1/\alpha} 
\phi\bigl(\log \bar v(t)\bigr)^{-1/\alpha}\\
&\sim \bar r(t)
\Bigl(e\, \alpha \,t\, \phi\bigl(\log \bar r(t) \log \bar r(t)\bigr)\Bigr)^{-1/\alpha}
\sim \Psi(t)\,.
\end{aligned}
$$
The third of the last asymptotic equivalences relies on the Uniform Convergence Theorem
for regularly varying functions \cite[Thm. 1.5.2]{BinGolTeu}. It will be used 
several times below without repeated mention.

\smallskip
\noindent
\textbf{Step 2. Asymptotic expansion near the maximum.} We decompose $\bar r(t)$
into integer and fractional part: $\bar r(t) = \bar k(t) + \bar \tau(t)$.
Our claim is that in \eqref{eq:Laplace}, the main contribution to the sum comes from
$\bar k(t)$ and $\bar k(t)+1$, that is, 
\begin{equation}\label{eq:main}
\mathfrak{p}(t) \sim 
%
%+ e^{-F(t\,,\,\bar k(t)+1)}%\bigl( 1 +  B(t) \bigr)\,,
\exp\Bigl(-F\bigl(t,\bar k(t)\bigr)\Bigr) + \exp\Bigl(-F\bigl(t,\bar k(t)+1\bigr)\Bigr)
\quad \text{as }\; t \to \infty\,.
\end{equation}
Indeed, we decompose the last the sum in \eqref{eq:Laplace}:
$$
\begin{aligned}
\mathfrak{p}(t) &\sim e^{-F(t\,,\,\bar k(t))} \bigl( 1 +  A(t) \bigr) 
+ e^{-F(t\,,\,\bar k(t)+1)} \bigl( 1 +  B(t) \bigr)\,, \quad \text{where}\\
A(t)&= \sum_{k=M}^{\bar k(t)-1}e^{F(t\,,\,\bar k(t))- F(t\,,\,k)} \quad\text{and}\quad
B(t)= \sum_{k=\bar k(t)+2}^{\infty}e^{F(t\,,\,\bar k(t)+1)- F(t\,,\,k)}
\end{aligned}
$$
and show that $A(t)\,,\; B(t) \to 0$ as $t \to \infty\,$. 

\smallskip

We start with $B(t)$. Using that $(\lambda_k)$ is strictly decreasing and 
$\bar k(t) + 1 > \bar r(t)$,
$$
\begin{aligned}
B(t) &= \sum_{k = \bar k(t) + 1}^{\infty} e^{t(\lambda_{\bar k(t)} - \lambda_{k})}
\frac{\bar k(t)!}{k!} \\
&\le \sum_{j=1}^{\infty} 
\frac{1}{\bigl(\bar k(t) + 1\bigr)\cdots \bigl(\bar k(t) + j\bigr)} \le 
\sum_{j=1}^{\infty} \frac{1}{\bar r(t)^j} \to 0\,,\quad \text{as }\; t \to \infty\,.
\end{aligned}
$$
We turn to $A(t)$ and use that by \eqref{eq:alpha-t}, 
$\;t \,\Lambda\bigl(\bar k(t)!\bigr) \ge t \,\Lambda\bigl(\bar v(t)\bigr) \ge 1/(2\alpha)$ 
for all large $t$. We also use that $e^{-u} \le \ell ! \,u^{-\ell}$ for all $u > 0$ and 
all $\ell \in \mathbb{N}$. Thus,
$$
\begin{aligned}
 A(t)&= \sum_{k=M}^{\bar k(t)-1} \exp \biggl( - t\,\Lambda\bigl(\bar k(t)!\bigr) 
\Bigl( \frac{\Lambda(k!)}{\Lambda\bigl(\bar k(t)!\bigr)} - 1\Bigr) \biggr) \,
\frac{\bigl(\bar k(t)-1\bigr)!}{(k-1)!}\\
&\le e^{1/(2\alpha)}\sum_{k=M}^{\bar k(t)-1} \exp \biggl( - \frac{1}{2\alpha} 
\frac{\Lambda(k!)}{\Lambda\bigl(\bar k(t)!\bigr)} \biggr) \,
\frac{\bar k(t)!}{k!}\\
&\le e^{1/(2\alpha)}\,\ell !\, (2\alpha)^{\ell} \sum_{k=M}^{\bar k(t)-1} 
\biggl(\frac{\Lambda\bigl(\bar k(t)!\bigr)}{\Lambda(k!)} \biggr)^{\!\!\ell}\,
\frac{\bar k(t)!}{k!}.
\end{aligned}
$$
We now apply Potter's bounds to our regularly varying function $\Lambda(\cdot)$ with index 
$-\alpha$, see \cite[Thm. 1.5.6]{BinGolTeu}: possibly at the cost of taking a bigger 
value for $M$, given $\varepsilon = \alpha/2$ there is $a > 0$ such that for all $k$ 
in our summation range,
$$
\frac{\Lambda\bigl(\bar k(t)!\bigr)}{\Lambda(k!)} 
\le a \left(\frac{\bar k(t)!}{k!}\right)^{\!-\alpha + \varepsilon}.
$$
Therefore
$$
A(t) \le e^{1/(2\alpha)}\,\ell !\,(2 c \alpha)^{\ell}\sum_{k=M}^{\bar k(t)-1} 
\left(\frac{\bar k(t)!}{k!}\right)^{\!1 -\ell\alpha/2}
\le  e^{1/(2\alpha)}\,\ell !\, (2 a \alpha)^{\ell}\,\, \bar k(t)^{2 -\ell\alpha/2}\,.
$$
If we choose $\ell > 4/\alpha$, this tends to $0$ as $t$ and thus also $\bar k(t)$
tends to $\infty\,$.

\smallskip
\noindent
\textbf{Step 3. The upper and lower asymptotic bounds.} 
Let $\Delta_1(t) = t\,\Lambda\bigl(\bar k(t)!\bigr)$ and 
$\Delta_2(t) = t\,\Lambda\bigl((\bar k(t)+1)!\bigr)$.
We have from \eqref{eq:Lambda}
$$
\begin{aligned} 
t^{1/\alpha}\,e^{-F(t\,,\,\bar k(t))} &= t^{1/\alpha}\,e^{-\Delta_1(t)} 
\, \bar k(t) \left(\frac{\Lambda\bigl(\Gamma(\bar k(t)+1)\bigr)}
{\phi\bigl(\log \Gamma(\bar k(t)+1)\bigr)}\right)^{1/\alpha}\\
&\sim \Delta_1(t)\, e^{- \Delta_1(t)} \,
\rho(t)\,\phi\bigl(\rho(t)\bigr)^{-1/\alpha} \quad \text{and}
\\[4pt]
t^{1/\alpha}\,e^{-F(t\,,\,\bar k(t)+1)} 
&\sim \Delta_2(t)\, e^{- \Delta_2(t)} \,
\rho(t)\,\phi\bigl(\rho(t)\bigr)^{-1/\alpha}.
\end{aligned}
$$
Therefore
\begin{equation}\label{eq:Delta}
\begin{aligned}
\mathfrak{p}(t) &\sim \Psi(t)\, (e\,\alpha)^{1/\alpha}
\Bigl(f\bigl(\Delta_1(t)\bigr) + f\bigl(\Delta_2(t)\bigr)\Bigr)\,,
\quad \text{where}\\
f(s) &= s^{1/\alpha}e^{-s}\,.
\end{aligned}
\end{equation}
Since $\Lambda(\cdot)$ is strictly decreasing, \eqref{eq:alpha-t} and
\eqref{eq:Lambda} yield that 
for large $t$,
\begin{equation}\label{eq:delta}
\Delta_2(t) < 1/\alpha \le \Delta_1(t) \quad\text{and}\quad
\delta(t) = \Delta_1(t)/\Delta_2(t) \sim \rho(t)^{\alpha}.
\end{equation}
The function $f(s)$ takes its unique maximum $(e\,\alpha)^{-1/\alpha}$
at $s = 1/\alpha$ and tends to $0$ at $0$ and $\infty\,$.

\smallskip
\noindent
\textbf{A. Upper bound.} 
Suppose that $t$ tends to infinity in such a way that $\Delta_1(t)$
remains bounded. Then $\Delta_2(t) \to 0$. Therefore
$$
f\bigl(\Delta_1(t)\bigr) + f\bigl(\Delta_2(t)\bigr)
\le (e\,\alpha)^{-1/\alpha} + f\bigl(\Delta_2(t)\bigr) 
\to (e\,\alpha)^{-1/\alpha}.
$$
On the other hand, suppose that  $t$ tends to infinity in such a
way that $\Delta_1(t) \to \infty\,$. Then 
$$
f\bigl(\Delta_1(t)\bigr) + f\bigl(\Delta_2(t)\bigr)
\le f\bigl(\Delta_1(t)\bigr) + (e\,\alpha)^{-1/\alpha} 
\to (e\,\alpha)^{-1/\alpha}.
$$
We infer that 
$\;
\limsup\limits_{t\to \infty} f\bigl(\Delta_1(t)\bigr) + f\bigl(\Delta_2(t)\bigr)
\le (e\,\alpha)^{-1/\alpha}.
$
\\[4pt]

To see that this is an equality, we observe that $t \mapsto \bar r(t)$ is
continuous and strictly increasing beyond some $t_0 > 0$. Therefore the 
(countable) set 
$T_0 = \{ t > t_0 : \bar r(t) \in \mathbb{N} \} =
 \{ t > t_0 : \bar r(t) = \bar k(t) \}$
is unbounded. If we let $t \in T_0$ tend to $\infty$, then
we know from Step 1 that $e^{-F(t\,,\,\bar k(t))} \sim \Psi(t)$,
so that \eqref{eq:main} and \eqref{eq:Delta} imply
$$
f\bigl(\Delta_1(t)\bigr) \sim (e\,\alpha)^{-1/\alpha} \quad\text{and}\quad
f\bigl(\Delta_2(t)\bigr) \to 0\,.
$$
The latter holds because by \eqref{eq:alpha-t} we are in a regime 
where $\Delta_1(t) \sim 1/\alpha$
remains bounded, compare with the above. We now conclude that
$$
\limsup_{t\to \infty} \mathfrak{p}(t)/\Psi(t) = 1\,.
$$
\smallskip
\noindent
\textbf{B. Lower bound.} We need to show that
\begin{equation}\label{eq:f-low}
\liminf_{t \to \infty} \frac{f\bigl(\Delta_1(t)\bigr) + f\bigl(\Delta_2(t)\bigr)}
{\bigl( \log \rho(t)\bigr)^{1/\alpha} \big/ \rho(t)} = 1\,.
\end{equation}
Recall \eqref{eq:delta}, and in particular the fact that $\Delta_1(t)$ and $\Delta_2(t)$ 
vary in $[1/\alpha\,,\,\infty)$ and $(0\,,\,1/\alpha)$, respectively. Thus
$$
\begin{aligned}
f\bigl(\Delta_1(t)\bigr) + f\bigl(\Delta_2(t)\bigr) 
&\ge  \max\bigl\{ f\bigl(\Delta_1(t)\bigr)\,,\, f\bigl(\Delta_2(t)\bigr\}
\\
&\ge \min \Bigl\{ \max\bigl\{ f\bigl(\delta(t) \,s\bigr) \,,\,f(s) \bigr\} : 
s \in \bigl[\tfrac{1}{\alpha\delta(t)}\,,\,\tfrac{1}{\alpha}\bigr) \Bigr\}
\end{aligned}
$$
On $\bigl[\tfrac{1}{\alpha\delta(t)}\,,\,\tfrac{1}{\alpha}\bigr)$, the function 
$s \mapsto f\bigl(\delta(t) \,s\bigr)$ is strictly decreasing, while 
$f(s)$ is strictly decreasing. Therefore, the last minimum is attained
in the unique point $s = s(t)$ which solves the equation $f\bigl(\delta(t) \,s\bigr) = f(s)$.
We easily compute %$f\bigl(s(t)\bigr)$
$$
s(t) = \frac{\delta(t)-1}{\log \delta(t)^{1/\alpha}} \sim \frac{\rho(t)^{\alpha}}{\log \rho(t)}
\quad \text{and}\quad
f\bigl(s(t)\bigr) %= \left( \frac{\delta(t)-1}{\log \delta(t)^{1/\alpha}} \right)^{-1/\alpha}
%\exp \left(-\frac{\delta(t)-1}{\log \delta(t)^{1/\alpha}} \right) 
\sim \frac{\bigl( \log \rho(t)\bigr)^{1/\alpha}}{\rho(t)}\,,
$$
as $t \to \infty\,$.
This shows that the $\liminf$ in \eqref{eq:f-low} is at least $1$. 
To verify that it is $=1$, we first claim that the set
$$
T_1 = 
\{ t > t_0 : \Delta_1(t) = \log \bar k(t) + \log\log \bar k(t) + \log \alpha\}
$$
is unbounded. To see this, we recall that by Stirling's formula, 
$\Gamma(r+\tau+1) \sim r^{\tau}\, \Gamma(r+1)$ as $r \to \infty\,$, uniformly
for $\tau$ in any bounded interval. Using this, \eqref{eq:Lambda} and 
\eqref{eq:alpha-t}, 
$$
\frac{1}{\Gamma\bigl(\bar k(t) +1\bigr)} \sim 
\frac{\bar k(t)^{\bar\tau(t)}}{\bar v(t)} = 
\frac{\Lambda\bigl(\bar v(t)\bigr)^{1/\alpha}\,\bar k(t)^{\bar\tau(t)}}
{\phi\bigl(\bar v(t)\bigr)^{1/\alpha}} 
\sim \frac{\bar k(t)^{\bar\tau(t)}}
{(\alpha\,t)^{1/\alpha}\,\phi\bigl(\bar r(t) \log \bar r(t)\bigr)^{1/\alpha}},
$$
so that
$$
\Delta_1(t) \sim   \frac{\phi\bigl(\Gamma(\bar k(t)+1)\bigr)}
                        {\phi\bigl(\bar r(t) \log \bar r(t)\bigr)} 
\,\alpha^{-1} \bar k(t)^{\alpha \, \bar\tau(t)}\sim  
\alpha^{-1} k(t)^{\alpha \, \bar\tau(t)} \,,\quad \text{as }\; t \to \infty.
$$
Thus, the defining equation for the set $T_1$ becomes
$$
\alpha^{-1}\bar k(t)^{\alpha \, \bar\tau(t)}\bigl(1 + o(1)\bigr) 
= \log \bar k(t) + \log\log \bar k(t) + \log \alpha\,,
$$
which transforms into 
\begin{equation}\label{eq:tau}
\bar \tau(t) = \frac{\log\log \bar k(t)}{\alpha \log \bar k(t)}\bigl(1 + o(1)\bigr). 
\end{equation}
At this point, we recall from above the set $T_0 = \{ t > t_0 : \bar \tau(t) = 0\}$, 
which is discrete, countable and unbounded,
so that it has the form  $\{t_j : j \in \mathbb{N}\}$ with $t_j < t_{j+1} \to \infty\,$.
Since $\bar r(t)$ is continuous and strictly increasing, the function 
$\bar\tau:  [t_j\,,\,t_{j+1}) \to [0,1)$ is continuous, strictly increasing and surjective.
Therefore, if $t_j$ is sufficiently large, the equation \eqref{eq:tau} has precisely
one solution in $[t_j\,,\,t_{j+1})$, and $T_1$ is unbounded. 

\smallskip

Now we can let $t \to \infty$ within $T_1\,$,
and then
$$
\frac{f\bigl(\Delta_1(t)\bigr)}{f\bigl(\Delta_2(t)\bigr)} = 
\exp\Bigl( 
\log\bigl( \delta(t)^{1/\alpha}\bigr)-\Delta_1(t)\bigl(1- 1/\delta(t)\bigr)  
\Bigr) \to 0\,,
$$
because we have $\delta(t) \sim \bar k(t)^{\alpha}$, so that
$$
\log\bigl( \delta(t)^{1/\alpha}\bigr)-\Delta_1(t)\bigl(1- 1/\delta(t)\bigr)
\sim
- \log\log \bar k(t)\,.
$$ 
We also see that $\Delta_2(t) \to 0$, so that
$$
f\bigl(\Delta_1(t)\bigr) + f\bigl(\Delta_2(t)\bigr) \sim f\bigl(\Delta_2(t)\bigr)
\sim \Delta_2(t)^{1/\alpha} 
\sim \frac{\bigl(\log \rho(t)\bigr)^{1/\alpha}}{\rho(t)}\,, 
$$
as $t \to \infty$ within $T_1\,$. This proves \eqref{eq:f-low}.
\end{proof}

\begin{comment}
\begin{example}
Following \cite{BrofWoess2001} and \cite{BenLSC2016} assume that
$\lambda(k)\sim(k!)^{-\alpha}k^{\gamma}$, $\gamma\in\mathbb{R}$, equivalently
\[
\lambda(V)=V^{-\alpha}\phi(\ln V),\quad\mathrm{where}\ \ \phi(\tau)\sim\left(
\frac{\tau}{\ln\tau}\right)  ^{\gamma},
\]
then we get
\begin{align*}
\Psi(t)  &  =\left(  \frac{1}{\alpha}\right)  ^{1-\gamma/\alpha}\left(
\frac{1}{e\alpha}\right)  ^{1/\alpha}t^{-1/\alpha}\left(  \frac{\ln t}{\ln\ln
t}\right)  ^{1-\gamma/\alpha},\\
\psi(t)  &  =\left(  \frac{1}{\alpha}\right)  ^{-\gamma/\alpha}t^{-1/\alpha
}\frac{\left(  \ln t\right)  ^{-\gamma/\alpha}}{\left(  \ln\ln t\right)
^{-(\gamma+1)/\alpha}}%
\end{align*}
and
\[
\limsup_{t\rightarrow\infty}\frac{p(t)}{\Psi(t)}=1\text{ \ and \ \ }%
\liminf_{t\rightarrow\infty}\frac{p(t)}{\psi(t)}=1.
\]

\end{example}
\end{comment}

\bigskip

\noindent Alexander Bendikov, Instytut Matematyczny, Uniwersytet Wrocławski,\\
Pl. Grunwaldzki 2/4, 50-384 Wrocław, Poland\\
e-mail: bendikov@math.uni.wroc.pl

\medskip
\noindent Wojciech Cygan, Institut für Diskrete Mathematik, Technische Universität Graz,
Steyrergasse 30, A-8010 Graz, Austria\\
Instytut Matematyczny, Uniwersytet Wrocławski, Pl. Grunwaldzki 2/4, 50-384 Wrocław, Poland\\
e-mail: wojciech.cygan@uwr.edu.pl

\medskip
\noindent Wolfgang Woess, Institut für Diskrete Mathematik, Technische Universität Graz,
Steyrergasse 30, A-8010 Graz, Austria\\
e-mail: woess@tugraz.at

\end{document}